\documentclass[a4, 12pt, 
]{amsart}

\usepackage[all]{xy}
\usepackage{mathrsfs, amsthm, amscd, amsmath, amssymb, graphicx}
\usepackage{float}
\usepackage{graphicx}

\usepackage[dvipdfmx]{hyperref}

\newtheorem{theo}{Theorem}[section]
\newtheorem{prop}[theo]{Proposition}
\newtheorem{lemm}[theo]{Lemma}

\newtheorem{claim}[theo]{Claim}
\newtheorem{prob}[theo]{Problem}

\numberwithin{equation}{section}

\theoremstyle{definition}
\newtheorem{defi}[theo]{Definition}

\theoremstyle{remark}
\newtheorem{rem}[theo]{Remark}

\newcommand{\Ker}[0]{\operatorname{Ker}}

\newcommand{\Hom}[0]{\operatorname{Hom}}
\newcommand{\Alb}[0]{\operatorname{Alb}}
\newcommand{\Pic}[0]{\operatorname{Pic}}
\newcommand{\tr}[0]{\operatorname{tr}}

\newcommand{\End}[0]{\operatorname{End}}
\newcommand{\rank}[0]{\operatorname{rank}}
\newcommand{\codim}[0]{\operatorname{codim}}
\newcommand{\Sym}[0]{\operatorname{Sym}}
\newcommand{\GL}[0]{\operatorname{GL}}

\newcommand{\deldel}{\sqrt{-1}\partial \overline{\partial}}

\usepackage{geometry}
\begin{document}

\title[On projective manifolds with pseudo-effective tangent bundle]
{On projective manifolds with \\ pseudo-effective tangent bundle}

\author{Genki HOSONO}
\address{Mathematical Institute, Tohoku University, 
	6-3, Aramaki Aza-Aoba, Aoba-ku, Sendai 980-8578, Japan.}
\email{{\tt genki.hosono@gmail.com}}

\author{Masataka IWAI}
\address{Graduate School of Mathematical Sciences, The University of Tokyo, 3-8-1 Komaba,
Tokyo, 153-8914, Japan.}
 \email{{\tt
masataka@ms.u-tokyo.ac.jp, masataka.math@gmail.com}}

\author{Shin-ichi MATSUMURA}
\address{Mathematical Institute, Tohoku University, 
6-3, Aramaki Aza-Aoba, Aoba-ku, Sendai 980-8578, Japan.}
\email{{\tt mshinichi-math@tohoku.ac.jp, mshinichi0@gmail.com}}

\date{\today, version 0.01}

\renewcommand{\subjclassname}{%
\textup{2010} Mathematics Subject Classification}


\maketitle

\begin{abstract}
In this paper, we develop the theory of singular hermitian metrics on vector bundles. 
As an application, we give a structure theorem of a projective manifold $X$ with pseudo-effective tangent bundle: 
$X$ admits a smooth fibration $X \to Y$ to a flat projective manifold $Y$ 
such that its general fiber is rationally connected. 
Moreover, by applying this structure theorem, 
we classify all the minimal surfaces with pseudo-effective tangent bundle and study general non-minimal surfaces, 
which provide examples of (possibly singular) positively curved tangent bundles. 
\end{abstract}

\tableofcontents

\section{Introduction}\label{Sec-1}
The structure theorem for compact K\"ahler manifolds 
with semi-positive bisectional curvature was established by 
Howard-Smyth-Wu and Mok in \cite{HSW81} and \cite{Mok88}
after the Frankel conjecture (resp. the Hartshorne conjecture)  
had been solved by Siu-Yau (resp. Mori) in \cite{SY80} (resp. \cite{Mor79}). 
As an algebraic analog of semi-positive bisectional curvature,  
Campana-Peternell and Demailly-Peternell-Schneider 
generalized the structure theorem of Howard-Smyth-Wu to nef tangent bundles in \cite{CP91} and \cite{DPS94}, 
and further they classified the surfaces and the $3$-folds with nef tangent bundle.  
(see  \cite{CP91} and \cite{MOSWW} for the Campana-Peternell conjecture).

It is of interest to consider pseudo-effective tangent bundles as a natural generalization of the above structure results. 
The theory of singular hermitian metrics on vector bundles, which has been rapidly developed, 
is a crucial tool to understand pseudo-effective vector bundles. 
Therefore, in this paper, we first develop the theory of singular hermitian metrics on vector bundles  (more generally torsion free sheaves). 
As one of the main applications, 
we obtain the following structure theorem 
for projective manifolds with pseudo-effective tangent bundle 
(and also for compact K\"ahler manifolds, see Theorem \ref{theo-posi-Ka}).

\begin{theo}\label{theo-posi}
Let $X$ be a projective manifold with pseudo-effective tangent bundle. 
Then $X$ admits a $($surjective$)$ morphism $\phi: X \to Y$ with connected fiber to 
a smooth manifold $Y$ with the following properties\,$:$
\begin{itemize}
\item[(1)] The morphism $\phi: X \to Y$ is smooth $($that is, all the fibers are smooth$)$. 
\item[(2)] The image $Y$ admits a finite \'etale cover $A \to Y$ by an abelian variety $A$. 
\item[(3)] A general fiber $F$ of $\phi$ is rationally connected. 
\item[(4)] A general fiber $F$ of $\phi$ also has the pseudo-effective tangent bundle. 
\end{itemize}
Moreover, if we further assume that 
$T_X$ admits a positively curved singular hermitian metric, 
then we have$:$
\begin{itemize}
\item[(5)] The standard exact sequence of tangent bundles 
$$
0 \longrightarrow T_{X/Y} 
\longrightarrow T_X \longrightarrow \phi^* T_Y \longrightarrow 0
$$
splits. 
\item[(6)] The morphism $\phi: X \to Y$ is locally trivial $($that is, all the fibers are smooth and isomorphic$)$. 
\end{itemize}
\end{theo}

Theorem \ref{theo-posi} is based on the argument in \cite{Mat18} and 
the theory of singular hermitian metrics on vector bundles developed in this paper. 
In particular, Theorem \ref{theo-vect}, Theorem \ref{theo-split2}, 
and Theorem \ref{theo-split} play an important role in the proof. 
Theorem \ref{theo-vect}, which can be seen as a generalization of \cite{CM}, 
gives a characterization of numerically flat vector bundles in terms of pseudo-effectivity. 
The proof depends on the theory of admissible Hermitian-Einstein metrics in \cite{BS}. 
Theorem \ref{theo-split2} and Theorem \ref{theo-split} were proved in \cite{HPS18} 
under the stronger assumption of the minimal extension property. 
Our contribution is to remove this assumption, 
which enables us to use the notion of singular hermitian metrics flexibly.

\begin{theo}\label{theo-vect}
Let $X$ be a projective manifold and let $\mathcal{E}$ be a reflexive coherent sheaf on $X$. 
If $\mathcal{E}$ is pseudo-effective and the first Chen class $c_1(\mathcal{E})$ is zero, 
then $\mathcal{E}$ is locally free on $X$ and numerically flat. 
\end{theo} 

\begin{theo}\label{theo-split2}
Let $E$ be a vector bundle with positively curved $($singular$)$ hermitian metric 
on a $($not necessarily compact$)$ complex manifold $X$. 
Let 
$$
0 \to S \to E \to Q \to 0 
$$
be an exact sequence by vector bundles $S$ and $Q$ on $X$. 
If the first Chern class $c_{1}(Q)$ is zero, 
the above exact sequence splits. 
\end{theo}

\begin{theo}\label{theo-split}
Let $X$ be a  compact K\"ahler manifold and 
let 
$$
0 \to \mathcal{S} \to \mathcal{E} \to \mathcal{Q} \to 0 
$$
be an exact sequence of reflexive coherent sheaves $\mathcal{S}$, $\mathcal{E}$, and $\mathcal{Q}$ on $X$. 
If $ \mathcal{E}$ admits a positively curved $($singular$)$ hermitian metric 
and the first Chen class $c_1(\mathcal{Q} )=0$, 
then we have:
\begin{itemize}
\item[(1)] $\mathcal{Q} $ is locally free and hermitian flat.  
\item[(2)] $\mathcal{E} \to \mathcal{Q}$ is a surjective bundle morphism on $X_\mathcal{E}$. 
\item[(3)] The above exact sequence splits on $X$. 
\end{itemize}
Here $X_\mathcal{E}$ is the maximal Zariski open set where $\mathcal{E}$ is locally free. 
\end{theo}

It is natural to attempt to classify all the surfaces $X$ with pseudo-effective tangent bundle, 
as an application of Theorem \ref{theo-posi}. 
In the case of the tangent bundle being nef, 
a surface $X$ has no curve with negative self-intersection, 
and thus $X$ is always minimal. 
However, a surface $X$ with pseudo-effective tangent bundle may not be minimal, 
which is one of the difficulties to classify them. 
In this paper, we classify all the minimal surfaces (see subsection \ref{Sec4-1} for more detail): 
\begin{theo}\label{theo-mini}We have$:$
\begin{itemize}
\item[(1)] If a $($not necessarily minimal$)$ ruled surface $X \to C$ has the pseudo-effective tangent bundle $T_X$, 
then the base $C$ is the projective line $\mathbb{P}^1$ or an elliptic curve.
\item[(2)] Further, in the case of $C$ being an elliptic curve, 
the surface $X$ is  a minimal ruled surface $($that is, the ruling $X \to C$ is a smooth morphism$)$. 
\item[(3)] Conversely, any minimal ruled surfaces $X \to C$ 
over an elliptic curve and over the projective line $C=\mathbb{P}^1$ 
have the pseudo-effective tangent bundle $T_X$. 
\end{itemize}
\end{theo}

Moreover, we study the remaining problem (that is, the classification for blow-ups of Hirzebruch surfaces) in detail. 
These studies provide interesting examples of pseudo-effective or singular positively curved vector bundles.

\subsection*{Acknowledgements}
G.H. is supported by the  Grant-in-Aid for JSPS Fellows $\sharp$19J00473.
M.I. is  supported by the  Grant-in-Aid for JSPS Fellows $\sharp$17J04457
and by the Program for Leading Graduate Schools, MEXT, Japan. 
S.M. is supported by the Grant-in-Aid 
for Young Scientists (A) $\sharp$17H04821 from JSPS.

\section{Preliminaries}\label{Sec-2}
\subsection{Singular hermitian metrics on torsion free sheaves}\label{Sec2-1}
In this subsection, we recall the notion of singular hermitian metrics and positivity of vector bundles. 
Throughout this paper, we adopt the definition of singular hermitian metrics in \cite{HPS18}.

Let $\mathcal{E}$ be a torsion free sheaf on a complex manifold $X$, 
and further let $X_{\mathcal{E}}$ denote the maximal Zariski open set where  
$\mathcal{E}$ is locally free. 
In this paper, 
we denote ${{\Hom}}(\mathcal{E}, \mathcal{O}_X)$ by $\mathcal{E}^\vee$.
A \textit{singular hermitian metric} $g$ on $\mathcal{E}$ is 
a singular hermitian metric defined on the vector bundle $\mathcal{E}|_{ X_{\mathcal{E}}}$ 
(see \cite[Definition 16.1]{HPS18}), 
and it is said to be \textit{positively curved} if 
$\log |u|_{g^{\vee}}$ is a psh function 
for any local section $u$ of $\mathcal{E}|_{ X_{\mathcal{E}} }^\vee$, 
where $g^{\vee}$ is the dual metric defined by $g^{\vee}  = {}^t\! g^{-1}$. 

\begin{defi}\label{defi-pseudo}
A torsion free coherent sheaf $\mathcal{E}$ on a compact complex manifold $X$ 
is said to be \textit{pseudo-effective} if for any integer $m>0$ there exists 
a singular hermitian metric $h_m$ on $\Sym ^m \mathcal{E}$ such that 
$$
\deldel \log |u|^2_{h_m^{\vee}} \geq - \omega  \text{ on } X_{\mathcal{E}}
$$
for any local holomorphic section $u$ of $\Sym^m\mathcal{E}$. Here $\omega$ is a fixed hermitian form on $X$. 
\end{defi}

We summarize the notions of positivity of vector bundles and torsion free coherent sheaves.
The above definition is equivalent to the definition (5) below 
when $X$ is a projective manifold (see \cite[Theorem 1.3]{Iwa}).

\begin{defi}[{\cite[Definition 7.1]{BDPP},\cite[Definition 1.17]{DPS94}, \cite[Definition 6.4]{DPS01}, \cite[Definition 3.20]{Nak04}}]

Let $X$ be a projective manifold.
 \begin{enumerate}
 \item A vector bundle $E$ is {\it nef} if $\mathcal{O}_{\mathbb{P}(E) }(1)$ is a nef line bundle on $\mathbb{P}(E) $.
 \item A vector bundle $E$ is {\it numerically flat} if $E$ is nef and $c_{1}(E) = 0$.
  \item A vector bundle $E$ is {\it almost nef} if there exists a countable family  of proper subvarieties $Z_{i}$ of $X$ 
  such that $E |_{C} $ is nef for any curve $ C \not \subset \cup_{i} Z_i$.
 \item A torsion free coherent sheaf $\mathcal{E}$ is {\it weakly positive at $x \in X$} if, 
 for any $a \in \mathbb{Z}_{+}$ and for any ample line bundle $A$ on $X$, there exists 
$b \in \mathbb{Z}_{+}$ such that $\Sym^{ab}( \mathcal{E}  ) ^{\vee \vee} \otimes A^{ b}$ is globally generated at $x$, where $\Sym^{ab}( \mathcal{E}  ) ^{\vee \vee}$
is the double dual of $ab$-th symmetric power of $ \mathcal{E} $.
 \item A torsion free coherent sheaf $\mathcal{E}$ is {\it pseudo-effective} if  $ \mathcal{E}  $ is weakly positive at some $x \in X$.
 \item A torsion free coherent sheaf $\mathcal{E}$ is {\it big}  if there exist $a \in \mathbb{Z}_{+}$ and 
 an ample line bundle $A$ on $X$ such that $\Sym^{a}( \mathcal{E}  ) ^{\vee \vee} \otimes A^{-1}$ is pseudo-effective.
 \item A torsion free coherent sheaf $\mathcal{E}$ is {\it generically globally generated} if $\mathcal{E}$ is globally generated 
 at a general point in $X$.
 \end{enumerate}
 \end{defi}
The definition of nef (resp. big,  or pseudo-effective) vector bundles coincides with 
the usual one in the case $E$ being a line bundle. 
Relationships among them can be summarized  by the following table: 

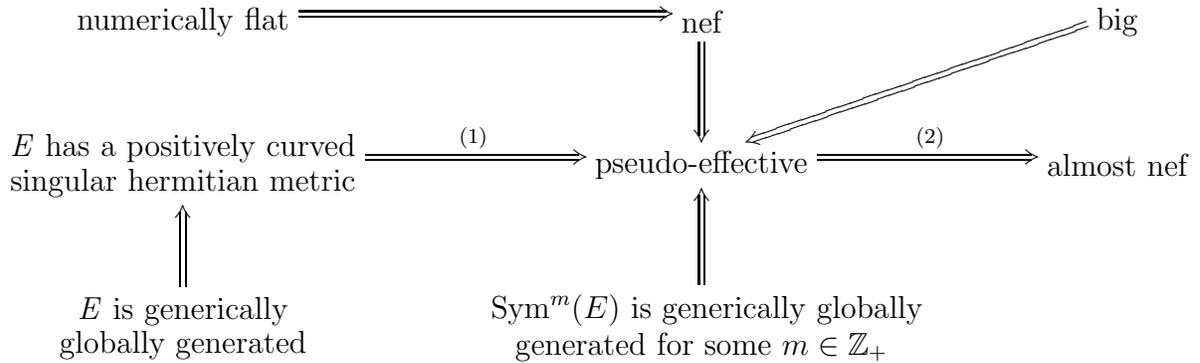
\begin{equation*}
\xymatrix@C=40pt@R=30pt{
  \txt{numerically flat} \ar@{=>}[r]& \txt{nef} \ar@{=>}[d] & \txt{big} \ar@{=>}[ld] \\ 
 \txt{\textit{E} has a positively curved \\ singular hermitian metric} \ar@{=>}[r]^{\ \ \ \ \ \  (1) }  &  \txt{pseudo-effective} \ar@{=>}[r]^{\ \ \ \  (2) }& \txt{almost nef}\\   
  \txt{\textit{E} is generically \\ globally generated} \ar@{=>}[u]  &\txt{  $ \Sym^{m}( E )$ is generically globally \\ generated  for some $m \in \mathbb{Z}_{+}$} \ar@{=>}[u]&\\
}
\end{equation*}

When $E$ is a line bundle, the converses of (1) and (2) hold (see \cite[Section 6]{Dem} and \cite[Theorem 0.2]{BDPP}). 
However, in higher rank case, the converse of (1) is not always true 
(see \cite[Example 5.4]{Hos17}) and the converse of (2) is unknown.

\section{Proof of the main results}\label{Sec-3}
This section is devoted to the proof of the main results. 

\subsection{Numerically flat vector bundles}\label{Sec-3-0}
In this subsection, we give a proof for Theorem \ref{theo-vect} 
after we prove Lemma \ref{lemm-nonvanish} and Lemma \ref{lemm-ele} for preparation. 
Lemma \ref{lemm-nonvanish}, which easily follows from the result of \cite[Proposition 1.16]{DPS94}, 
is quite useful and often used in this paper. 

\begin{lemm}\label{lemm-nonvanish}
Let $X$ be a projective manifold and 
let $\mathcal{E}$ be an almost nef torsion free coherent sheaf on $X$. 

\begin{itemize}
\item[(1)] Any non-zero section $\tau \in H^{0}(X, \mathcal{E}^{\vee})$ is 
non-vanishing on $X_\mathcal{E}$. 

\item[(2)] Let $\mathcal{S}$ be a reflexive coherent sheaf such that $\det \mathcal{S}$ is pseudo-effective and let $0 \to \mathcal{S} \to \mathcal{E}^\vee$ be an injective sheaf morphism. 
Then $\mathcal{S}$ is locally free on $X_\mathcal{E}$ and 
the above morphism is an injective bundle morphism on $X_\mathcal{E}$. 
\end{itemize}
\end{lemm}
\begin{proof}
In \cite{DPS94}, the same conclusion was proved for nef vector bundles. 
We denote by $Z$ a countable union of proper subvarieties of $X$ 
satisfying the definition of almost nef sheaves. 
We may assume that $X\setminus X_{\mathcal{E}} \subset Z$
by adding the subvariety $X\setminus X_{\mathcal{E}}$ into $Z$. 
\smallskip\\
(1)
Let $\tau \in H^{0}(X, \mathcal{E}^{\vee})$ be a non-zero section. 
For an arbitrary point $p \in X_{\mathcal{E}}$, 
by taking a complete intersection of ample hypersurfaces, 
we construct a curve $C$ passing through $p$ 
such that $ C \not \subset Z$.
We may assume that  $C \subset X_{\mathcal{E}}$ 
by $\codim (X \setminus X_{\mathcal{E}}) \geq 2$.  
Then $\mathcal{E}|_C$ is a nef vector bundle thanks to $C \subset X_{\mathcal{E}}$, 
and thus it follows that the non-zero section $\tau|_C$ is non-vanishing from \cite[Proposition 1.16]{DPS94}. 
In particular, the section $\tau$ is non-vanishing at $p$. 
\smallskip\\
(2) Following the argument in \cite{DPS94}, 
we obtain the non-zero section 
$$
\tau \in H^{0}(X, \Lambda^p  \mathcal{E}^{\vee} \otimes \det \mathcal{S}^\vee) 
$$ 
from the induced morphism $\det \mathcal{S} \to \Lambda^p \mathcal{E}^{\vee}$. 
Here $p:=\rank \mathcal{S}$. 
We remark that $\Lambda^p  \mathcal{E} \otimes \det \mathcal{S}$ is also almost nef by the assumption on $\mathcal{S}$. 
Hence, by applying the first conclusion and \cite[Lemma 1.20]{DPS94} to $\tau$, 
we can obtain the desired conclusion. 
\end{proof}

\begin{lemm}\label{lemm-nonvanish2}
Let $X$ be a compact complex manifold and 
let $\mathcal{E}$ be a pseudo-effective torsion free coherent sheaf on $X$. 
Then the same conclusion as in Lemma \ref{lemm-nonvanish} holds.  
\end{lemm}

\begin{proof}[Proof of Lemma \ref{lemm-nonvanish2}]
We will prove only the conclusion (1). 
For the metric $h_m$ on $\Sym^m \mathcal{E}$ satisfying the property 
in Definition \ref{defi-pseudo}, 
we consider the function $f_m$ on $X$ defined by 
$$
f_m:=\frac{1}{m} \log | \tau^{ m} |_{h^\vee_m}. 
$$
By the construction of $h_m$, we have 
$$
\deldel f_m \geq -\frac{1}{m} \omega, 
$$
and thus its weak limit (after we take a subsequence) should be zero. 
On the other hand, when we assume $\tau$ has the zero point at some point $p \in X_{\mathcal{E}}$, 
it can be shown that the Lelong number of $f_m$ is greater than or equal to one. 
This is a contradiction to the fact that the weak limit is zero. 
Indeed, the section $\tau^{  m}$ can be locally written as 
$
\tau^{m}=\sum_I \tau_I e_I. 
$
Here $\{e_i\}_{i=1}^r$ is a local frame of $\mathcal{E}$, 
$I$ is a multi-index of degree $m$, and $e_I:=\prod_{i\in I} e_i$. 
It follows that the holomorphic function $\tau_I$ has the multiplicity $\geq m$ at $p$
from $\tau=0$ at  $p \in X_{\mathcal{E}}$. 
It can be seen that $|\langle e_I, e_J \rangle_{h^\vee_m}|$ is bounded 
since 
$\log | u |_{h^\vee_m}$ is almost psh for any local section $u$
(for example see \cite[Lemma 2.2.4]{PT}). 
Hence we can easily check that 
$$
| \tau^{  m} |_{h^\vee _m} \leq C \sum_I |\tau_I |. 
$$
This implies that the Lelong number of $f_m$ is greater than or equal to one. 
\end{proof}

\begin{lemm}\label{lemm-ele}
Let $X$ be a projective manifold and $E$ be a vector bundle on $X$. 
Let $X_0$ be a Zariski open set in $X$ with $\codim (X \setminus X_0) \geq 2+i$.
Then the morphism induced by the restriction 
$$
H^{j}(X, E) \to H^{j}(X_0, E)
$$
is an isomorphism for any $j \leq i$. 
\end{lemm}
\begin{proof}
The proof is given by the standard argument in terms of ample hypersurfaces 
and the induction on dimension. 
\end{proof}

Theorem \ref{re-theo-vect}, which is a slight generalization of \cite{CM}, 
heavily depends on the theory of admissible Hermitian-Einstein metrics developed in \cite{BS}.

\begin{theo}[=Theorem \ref{theo-vect}, cf. \cite{CM}]\label{re-theo-vect}
Let $X$ be a projective manifold and 
let $\mathcal{E}$ be a reflexive coherent sheaf. 
If $\mathcal{E}$ is pseudo-effective and the first Chen class $c_1(\mathcal{E})$ is zero, 
then $\mathcal{E}$ is locally free and numerically flat. 
\end{theo} 

\begin{proof}[Proof of Theorem \ref{re-theo-vect}]
The induction on the rank $r$ of $\mathcal{E}$ will give the proof. 
Reflexive coherent sheaves of rank one are always line bundles (see \cite{Har80}), 
and thus the conclusion is obvious in the case of $r=1$. 
It is not so difficult to check the numerical flatness of $\mathcal{E}$ 
 if $\mathcal{E}$ is shown to be locally free (see the proof in \cite[Theorem 1.18]{DPS94} or the argument below). 
We will focus on the proof of local freeness.

In the proof, we fix an ample line bundle $A$ on $X$. 
In the case of $r>1$, we take a coherent subsheaf $\mathcal{S}$ 
with the minimal rank among coherent subsheaves of $\mathcal{E}$ satisfying that 
$\int_X c_{1} (\mathcal{S}) \cdot c_{1}(A)^{n-1} \geq 0$. 
We may assume that $\mathcal{S}$ is reflexive by taking the double dual if necessary. 
Now we consider the following exact sequence of sheaves:
\begin{align}\label{exact}
0 \rightarrow \mathcal{S} \rightarrow \mathcal{E} \rightarrow 
\mathcal{Q}:=\mathcal{E}/\mathcal{S}    \rightarrow 0.
\end{align}
The quotient sheaf $\mathcal{Q}:=\mathcal{E}/\mathcal{S} $ is pseudo-effective. 
In particular, the first Chern class $c_{1}(\mathcal{Q})$ is also pseudo-effective. 
On the other hand, we have 
$$
0=c_{1}(\mathcal{E})=c_{1}(\mathcal{S})+c_{1}(\mathcal{Q}). 
$$
Then it follows that $c_{1}(\mathcal{S}) = c_{1}(\mathcal{Q})=0$ 
since $c_{1}(\mathcal{Q})$ is pseudo-effective and we have 
$$
\int_X c_{1} (\mathcal{Q}) \cdot c_{1}(A)^{n-1}=
-\int_X c_{1} (\mathcal{S}) \cdot c_{1}(A)^{n-1}\leq 0. 
$$
By applying Lemma \ref{lemm-nonvanish} to $\mathcal{Q}^{\vee} \to \mathcal{E}^\vee$, 
we can see that $\mathcal{Q}$ (and thus $\mathcal{S}$) is a vector bundle on $X_{\mathcal{E}}$ 
and the above morphism is a bundle morphism on $X_{\mathcal{E}}$.



We first consider the case where the rank of $\mathcal{S}$ is equal to $r=\rank \mathcal{E}$. 
In this case, we obtain $\mathcal{S}=\mathcal{E}$. 
Indeed, it follows that $\mathcal{S} \cong \mathcal{E}$ on $X_{\mathcal{E}}$
since the bundle morphism $\mathcal{S} \to \mathcal{E}$ on $X_{\mathcal{E}}$ is an isomorphism. 
Then we can easily check $\mathcal{S}=\mathcal{E}$ by the reflexivity and $\codim(X \setminus X_{\mathcal{E}} ) \geq 3$. 
Further we can prove that 
$$
\int_X c_{2}(\mathcal{E}) \cdot c_1(A)^{n-2}=0. 
$$
Indeed, for a surface $S:=H_1 \cap H_2 \cap \dots \cap H_{n-2}$ in $X$
constructed by general members $H_i$ of  a complete linear system $A$, 
it follows that 
$\mathcal{E}|_S$ is a pseudo-effective vector bundle 
from $\codim (X \setminus X_{\mathcal E}) \geq 3$. 
Hence $\mathcal{E}|_S$ is numerically flat on $S$, 
and thus $c_{2}(\mathcal{E}|_S) =0$ (see \cite{DPS94} or \cite[Corollary 2.12]{CH17}). 
We can easily check that 
$$
\int_X c_{2}(\mathcal{E}) \cdot c_1(A)^{n-2} =
\int_S c_{2}(\mathcal{E}|_S) =0. 
$$
By the assumption of $c_1(\mathcal{E})=0$ and the result of \cite[Corollary 3]{BS},  
we can conclude that $\mathcal{E}$ is a hermitian flat vector bundle on $X$ 
from the stability of the reflexive sheaf $\mathcal{S}=\mathcal{E}$.
Therefore  $\mathcal{E}$ is locally free and numerically flat.

It remains to consider the case of $\rank \mathcal{S} < \rank \mathcal{E}$.
In this case, we consider the surjective bundle morphism 
$$
\Lambda^{m+1} \mathcal{E} \otimes \det \mathcal{Q}^\vee \to \mathcal{S}
$$
on $X_{\mathcal{E}}$. 
By $\codim (X \setminus X_{\mathcal{E}}) \geq 3$ and $c_1(\mathcal{Q})=0$, 
the reflexive sheaf $\mathcal{S}$ is pseudo-effective. 
Therefore we can conclude that $\mathcal{S}$  is a numerically flat vector bundle on $X$ 
by the induction hypothesis. 

On the other hand, 
the sheaf $\mathcal{Q}$ itself may not be a vector bundle, 
but, the reflexive hull $\mathcal{Q}^{\vee \vee}$ is a vector bundle on $X$ 
by the induction hypothesis. 
The extension class obtained from the exact sequence (\ref{exact}) on $X_\mathcal{E}$ 
can be extended to the extension class (defined on $X$) of 
$\mathcal{S}$ and $\mathcal{Q}^{\vee \vee}$  by Lemma \ref{lemm-ele}. 
The extended class determines the vector bundle whose restriction to  $X_\mathcal{E}$  corresponds to $\mathcal{E}$. 
This implies that $\mathcal{E}$ is a vector bundle by the reflexivity of $\mathcal{E}$. 
\end{proof}

\subsection{Splitting theorem for positively curved vector bundles}\label{Sec-3-1}
In this subsection, 
we prove Theorem \ref{theo-split2} 
and Theorem \ref{theo-split}.

\begin{lemm}\label{lem-vector}
Let $\mathcal{Q}$ be a reflexive coherent sheaf on a compact complex manifold $X$. 
If $\mathcal{Q}$ admits a positively curved singular hermitian metric $g_\mathcal{Q}$ and $c_1(\mathcal{Q}) = 0$, 
then  we have: 
\begin{itemize}
\item[(1)] $(\mathcal{Q}, g_\mathcal{Q})$ is hermitian flat on $X_{\mathcal{Q}}$.  
\item[(2)] If we further assume that $X$ is K\"ahler, then 
$\mathcal{Q}$ is a locally free sheaf on $X$ and $g_\mathcal{Q}$ extends to  a hermitian flat metric on $X$. 
\end{itemize}
\end{lemm}
\begin{proof}
(1) We follow the argument in \cite{CP17}. 
The following lemma proved by Raufi \cite{Rau15} is essential:

\begin{lemm}[{\cite[Thm 1.6]{Rau15}}]\label{lemm-Raufi}
Let $E$ be a holomorphic vector bundle and $h_E$ be a positively curved singular hermitian metric on $E$.
If the induced metric $\det h_E$ on the determinant bundle $\det E$ is non-singular 
$($that is, smooth metric$)$, 
then the curvature current $\sqrt{-1}\Theta_{h_E}$ of $h_E$ is well-defined as an $\End (E)$-valued $(1,1)$-form with measure coefficients.
\end{lemm}

In our situation $\det g_{\mathcal{Q}}$ is a positively curved singular hermitian metric on the determinant bundle $\det \mathcal{Q}$.
By $c_1(\mathcal{Q}) = 0$, the curvature $\sqrt{-1}\Theta_{\det g_{\mathcal{Q}}}$ of $\det g_{\mathcal{Q}}$ is identically zero on $X_{\mathcal{Q}}$. 
In particular, it can be seen that $\det g_{\mathcal{Q}}$ is non-singular. 
Then, by Raufi's result, the curvature current $\sqrt{-1}\Theta = \sqrt{-1}\Theta_{g_{\mathcal{Q}}}$ of $g_{\mathcal{Q}}$ is well-defined on $X_{\mathcal{Q}}$.

We locally write the curvature $\sqrt{-1}\Theta$ as
$$
\sqrt{-1}\Theta = \sum_{j,k,\alpha,\beta} \mu_{j \overline{k} \alpha \overline{\beta}}dz^j \wedge d\overline{z}^k e_\alpha \otimes e^{\vee}_\beta, 
$$
where $(z_1, \ldots, z_n)$ denotes a local coordinate and $e_1, \ldots, e_r$ denotes a local frame of $\mathcal{Q}$.
Then, by $0=\sqrt{-1}\Theta_{\det g_{\mathcal{Q}}} = \sqrt{-1}\tr \Theta_{g_{\mathcal{Q}}} $,  we obtain 
$$\sum_{j,k} \sum_\alpha \mu_{j \overline{k} \alpha \overline{\alpha}} dz^j \wedge d\overline{z}^k = 0.$$
Since $g_{\mathcal{Q}}$ is positively curved,
$$\sum_{j,k} \mu_{j \overline{k} \alpha \overline{\alpha}} dz^j \wedge d\overline{z}^k \geq 0 $$
for every $\alpha$. Then we have that
$\mu_{j \overline{k}\alpha \overline{\alpha}} = 0$ for every $j,k,\alpha$.

For every $\alpha$ and $\beta$, we have that
$${\rm Re} (\xi^\alpha \overline{\xi}^\beta \sum_{j,k} \mu_{j \overline{k}\alpha \overline{\beta}} v^j \overline{v}^k) \geq 0.$$
From this we can conclude that $\mu_{j \overline{k}\alpha \overline{\beta} }=0$ for every $j,k, \alpha, \beta$ and thus $\sqrt{-1}\Theta_{g_{\mathcal{Q}}} = 0$.

\smallskip

(2) It follows that $\mathcal{Q}$ is polystable from (1) and \cite[Theorem 3]{BS}. 
We have $\codim (X \setminus X_{\mathcal{Q}}) \ge 3$ and 
$(\mathcal{Q}, g_Q)$ is hermitian flat on $X_{\mathcal{Q}}$. 
Hence it can be shown that $c_{1}(\mathcal{Q})=0$ and $c_{2}(\mathcal{Q})=0$. 
We can see that $\mathcal{Q}$ is actually locally free and hermitian flat by  \cite[Theorem 4]{BS}. 
 \end{proof}

We prepare the following lemma for the proof of Theorem \ref{theo-split2}. 

\begin{lemm}\label{lem-flat}
Let $(E,h)$ be a hermitian flat vector bundle on a complex manifold $X$. Then for any point $x \in X$ and a basis $e_{1, x}, \ldots, e_{r,x}$ on the fiber $E_x$, there exists a local holomorphic frame $e_1, \ldots, e_r$ near $x$ such that $e_j(x) = e_{j,x}$ and $\langle e_i, e_j\rangle_h$ is constant.
\end{lemm}

\begin{proof}
Let $D$ be the Chern connection associated to $(E,h)$.
Then, by flatness, we can take a local frame $\{e_j\}$ around $x$ such that $D e_j \equiv 0$. We can assume that $e_j(x) =e_{j,x}$.
Since $D$ is compatible with $h$, we have that $d\langle e_i, e_j \rangle_h = \{ De_i, e_j \}_h + \{ e_i, D e_j \}_h = 0$, thus $\langle e_i, e_j \rangle_h$ is constant. Moreover, taking the $(0,1)$-part of $De_j \equiv 0$, we obtain that $\bar{\partial}e_j \equiv 0 $, which shows that $e_j$ is holomorphic.
\end{proof}

\begin{theo}$($=Theorem \ref{theo-split2}$)$\label{re-theo-split2}
Let $E$ be a vector bundle with positively curved $($singular$)$ 
hermitian metric $g$
on a $($not necessarily compact$)$ complex manifold $X$. 
Let 
$$
0 \to S \to E \to Q \to 0 
$$
be an exact sequence of vector bundles on $X$. 
If the first Chern class $c_{1}(Q)$ is zero, 
the above exact sequence splits. 
\end{theo} 

\begin{proof}[Proof of Theorem \ref{re-theo-split2}]
The following proof is a generalization of \cite[Theorem 5.1]{Hos17}.
We will work on dual bundles. By taking the dual, we have the following exact sequence
\begin{equation}\label{dual-split}
0 \to Q^{\vee} \to E^{\vee} \to S^{\vee} \to 0. 
\end{equation} 
Then we have a negatively curved singular hermitian metric $h^{\vee}$ whose restriction to $Q^{\vee}$ is flat by (the dual of) Lemma \ref{lem-vector} (1). Therefore, by Lemma \ref{lem-flat}, 
we can take a holomorphic orthonormal frame $(\kappa^\alpha_1, \ldots, \kappa^\alpha_q)$ of $Q^{\vee}$ on a small open set $U^\alpha$. Let $\epsilon^\alpha_j$ be the image of $\kappa^\alpha_j$ in $E^{\vee}$. Take $\epsilon^\alpha_{q+1}, \ldots, \epsilon^\alpha_{q+s}$ such that $(\epsilon^\alpha_1, \ldots, \epsilon^\alpha_{q+s})$ is a local frame of $E^{\vee}$. Let $\sigma^\alpha_j$ be the image of $\epsilon^\alpha_j$ in $S^{\vee}$. We remark that 
$(\sigma^\alpha_{q+1}, \ldots, \sigma^\alpha_{q+s})$ is a local frame of $S^{\vee}$.
We will write the transition function of $Q^{\vee}$ and $S^{\vee}$ as follows:
$$
\begin{array}{ccc}
\kappa^\alpha_1 &= &\Phi^{Q^{\vee}, \alpha \beta}_{1,1} \kappa^{\beta}_1 + \cdots + \Phi^{Q^{\vee}, \alpha \beta}_{1,q} \kappa^{\beta}_q, \\
&\vdots&\\
\kappa^\alpha_q &= &\Phi^{Q^{\vee}, \alpha \beta}_{q,1} \kappa^{\beta}_1 + \cdots + \Phi^{Q^{\vee}, \alpha \beta}_{q,q} \kappa^{\beta}_q, \\
\vspace{2mm}\\
\sigma^\alpha_{q+1} &=& \Phi^{S^{\vee}, \alpha \beta}_{q+1, q+1} \sigma^{\beta}_{q+1} + \cdots + \Phi^{S^{\vee}, \alpha \beta}_{q+1, q+s} \sigma^{\beta}_{q+s},\\
&\vdots&\\
\sigma^\alpha_{q+s} &=& \Phi^{S^{\vee}, \alpha \beta}_{q+s, q+1} \sigma^{\beta}_{q+1} + \cdots + \Phi^{S^{\vee}, \alpha \beta}_{q+s, q+s} \sigma^{\beta}_{q+s}.\\
\end{array}
$$
The transition functions for $E^\vee$ can be written as
$$
\begin{array}{ccccccccl}
    \epsilon^\alpha_1 &=& \Phi^{Q^{\vee},\alpha\beta}_{11} \epsilon^\beta_1 &+& \cdots &+& \Phi^{Q^{\vee},\alpha\beta}_{1q} \epsilon^\beta_q,&&\\
    &\vdots &&&&&&&\\
    \epsilon^\alpha_q &=& \Phi^{Q^{\vee},\alpha\beta}_{q1} \epsilon^\beta_1 &+& \cdots &+& \Phi^{Q^{\vee},\alpha\beta}_{qq} \epsilon^\beta_q,&& \vspace{2mm}\\
    \epsilon^\alpha_{q+1} &=& \Phi^{E^{\vee},\alpha\beta}_{q+1,1} \epsilon^\beta_1 &+& \cdots &+& \Phi^{E^{\vee},\alpha\beta}_{q+1,q} \epsilon^\beta_q &+&  \Phi^{S^{\vee},\alpha\beta}_{q+1,q+1} \epsilon^\beta_{q+1} + \cdots + \Phi^{S^{\vee},\alpha\beta}_{q+1,q+s} \epsilon^\beta_r,\\
    &\vdots&&&&&&&\\
    \epsilon^\alpha_{q+s} &=& \Phi^{E^{\vee},\alpha\beta}_{q+s,1} \epsilon^\beta_1 &+& \cdots &+& \Phi^{E^{\vee},\alpha\beta}_{q+s,q} \epsilon^\beta_q &+&  \Phi^{S^{\vee},\alpha\beta}_{q+s,q+1} \epsilon^\beta_{q+1} + \cdots + \Phi^{S^{\vee},\alpha\beta}_{q+s,q+s} \epsilon^\beta_{q+s}.\\
    \end{array}
$$
For short we will write the coefficient matrix as 
$$\Phi^{E^{\vee}, \alpha\beta} = 
    \begin{pmatrix}
    \Phi^{Q^{\vee},  \alpha\beta} & 0\\
    \Psi^{ \alpha\beta} & \Phi^{S^{\vee},  \alpha\beta}\\
    \end{pmatrix}.$$
    
Next, let $h^\alpha$ be the matrix
$$h^\alpha := 
    \begin{pmatrix}
    \langle \epsilon_1^\alpha, \epsilon_1^\alpha \rangle_h & \langle \epsilon_1^\alpha, \epsilon_2^\alpha \rangle_h & \cdots & \langle \epsilon_1^\alpha, \epsilon_{q+s}^\alpha \rangle_h \\
    \vdots & \ddots & & \vdots \\
    \langle \epsilon_{q+s}^\alpha, \epsilon_1^\alpha \rangle_h & \cdots & & \langle \epsilon_{q+s}^\alpha, \epsilon_{q+s}^\alpha \rangle_h\\
    \end{pmatrix}. $$
Note that the upper-left $q \times q$ -matrix is constant by the choice of $\epsilon^\alpha_1, \ldots, \epsilon^\alpha_q$.
Since $h$ is negatively curved, by \cite[Proposition 5.2]{Hos17}, coefficients of the lower-left $s \times q$-matrix is holomorphic (say $\phi^\alpha$). Then we can write as
$$ h^\alpha = 
    \begin{pmatrix}
    C^\alpha & \overline{\phi^\alpha}\\
    \phi^\alpha & *
    \end{pmatrix},
    $$
where $C^\alpha$ is a $q \times q$-matrix whose coefficients are constant on $U^\alpha$.
By the equality 
$$ h^\alpha = \Phi^{E^{\vee}, \alpha\beta} h^\beta  (\overline{{}^t\Phi^{E^{\vee}, \alpha\beta} }), $$
we have 
\begin{align*}   
   C^\alpha &= \Phi^{Q^{\vee}, \alpha \beta} C^\beta (\overline{{}^t \Phi^{Q^{\vee}, \alpha \beta}}), \\
  \phi^\alpha &= \Psi^{\alpha \beta} C^\beta  (\overline{{}^t\Phi^{Q^{\vee}, \alpha \beta}}) + \Phi^{S^{\vee}, \alpha \beta} \phi^\beta (\overline{{}^t\Phi^{Q^{\vee}, \alpha \beta}}). 
\end{align*}
From these equalities, it follows that
$$\phi^\alpha (C^\alpha)^{-1} = \Psi^{\alpha\beta}(\Phi^{Q^{\vee}, \alpha\beta})^{-1} + \Phi^{S^{\vee}, \alpha\beta} \phi^\beta (C^\beta)^{-1} (\Phi^{Q^{\vee}, \alpha\beta})^{-1}.$$

On the other hand, the extension class of the given exact sequence can be calculated as the cohomology class of the \v{C}ech 1-cocycle
$$\left\{\sum_{\lambda = q+1}^{q+s} \sum_{\mu=1}^{q} \Psi^{\alpha\beta}_{\lambda,\mu} \kappa^\beta_\mu \otimes (\sigma^\alpha_\lambda)^{\vee}  \in H^0(U_{\alpha\beta}, \mathcal{O}(Q^{\vee} \otimes S)) \right\}_{\alpha\beta}$$
$$= \left\{\sum_{\lambda = q+1}^{q+s} \sum_{\mu=1}^{q} \sum_{\nu=1}^{q} \Psi^{\alpha\beta}_{\lambda,\mu}((\Phi^{Q^{\vee},\alpha\beta})^{-1})_{\mu\nu} \kappa^\alpha_\nu \otimes (\sigma^\alpha_\lambda)^{\vee} \right\}_{\alpha\beta}.$$
It is the differential of the following \v{C}ech 0-cochain
$$\left\{ \sum_{\nu=1}^{q} \sum_{\lambda = q+1}^{q+s} (\phi^\alpha (C^\alpha)^{-1})_{\lambda \nu} \kappa^\alpha_\nu \otimes (\sigma^\alpha_\lambda)^{\vee}  \in H^0(U_\alpha, \mathcal{O}(Q^{\vee}\otimes S))
    \right\}_\alpha,$$
thus the extension class is zero. 
Therefore the given sequence (\ref{dual-split}) splits.
\end{proof}

\begin{theo}[=Theorem \ref{theo-split}]\label{re-theo-split}
Let $X$ be a  compact complex manifold and 
let 
$$
0 \to \mathcal{S} \to \mathcal{E} \to \mathcal{Q} \to 0 
$$
be an exact sequence of reflexive coherent sheaves $\mathcal{S}$, $\mathcal{E}$, and $\mathcal{Q}$ on $X$. 
If $ \mathcal{E}$ admits a positively curved $($singular$)$ hermitian metric 
and the first Chen class $c_1(\mathcal{Q})=0$, 
then we have:
\begin{itemize}
\item[(1)] $\mathcal{Q}$ is locally free and hermitian flat.  
\item[(2)] $\mathcal{E} \to \mathcal{Q}$ is a surjective bundle morphism on $X_\mathcal{E}$. 
\item[(3)] The above exact sequence splits on $X$. 
\end{itemize}
\end{theo} 

\begin{proof}[Proof of Theorem \ref{re-theo-split}]
The conclusion (1) follows from Lemma \ref{lem-vector} and 
the conclusion (2) follows from Lemma \ref{lemm-nonvanish}. 
Also, from Theorem \ref{re-theo-split2}, it follows that there exists 
a bundle morphism $j:\mathcal{Q} \to \mathcal{E}$ on $X_{\mathcal{E}}$
such that 
$$
\mathcal{E}=\mathcal{S} \oplus j(\mathcal{Q}) \text{ on } X_{\mathcal{E}}. 
$$
By taking the pushforward $i_{*}$ by the natural inclusion 
$i: X_{\mathcal{E}} \to X$ and the double dual, 
we obtain 
$$
(i_*\mathcal{E})^{\vee \vee}=
(i_*\mathcal{S})^{\vee \vee} \oplus (i_*j(\mathcal{Q}))^{\vee \vee} \text{ on } X. 
$$
By $\codim (X \setminus X_{\mathcal{E}}) \geq 3$ and 
the reflexivity, 
we have $\mathcal{E} \cong (i_*\mathcal{E})^{\vee \vee}$, 
$\mathcal{S} \cong (i_*\mathcal{S})^{\vee \vee}$, and 
$\mathcal{Q} \cong (i_*j(\mathcal{Q}))^{\vee \vee}$. 
This finishes the proof. 
\end{proof}

\subsection{Pseudo-effective tangent bundles}\label{Sec-3-2}

This subsection is devoted to the proof of Theorem \ref{theo-posi}. 

\begin{theo}[=Theorem \ref{theo-posi}]\label{re-theo-posi}
Let $X$ be a projective manifold with pseudo-effective tangent bundle. 
Then $X$ admits a morphism $\phi: X \to Y$ with connected fiber to 
a smooth manifold $Y$ with the following properties\,$:$
\begin{itemize}
\item[(1)] The morphism $\phi: X \to Y$ is smooth $($that is, all the fibers are smooth$)$. 
\item[(2)] The image $Y$ admits a finite \'etale cover $A \to Y$ by an abelian variety $A$. 
\item[(3)] A general fiber $F$ of $\phi$ is rationally connected. 
\item[(4)] A general fiber $F$ of $\phi$ also has the pseudo-effective tangent bundle. 
\end{itemize}
Moreover, if we further assume that 
$T_X$ admits a positively curved singular hermitian metric, 
then 
\begin{itemize}
\item[(5)] The following exact sequence splits: 
$$
0 \longrightarrow T_{X/Y} 
\longrightarrow T_X \longrightarrow \phi^* T_Y \longrightarrow 0. 
$$
\item[(6)] The morphism $\phi: X \to Y$ is locally trivial $($that is, all the fibers are smooth and isomorphic$)$. 
\end{itemize}
\end{theo}

\begin{proof}[Proof of Theorem \ref{re-theo-posi}]

For a projective manifold $X$ with the pseudo-effective tangent bundle $T_X$, 
we consider an MRC fibration $\phi: X \dashrightarrow Y$ to a projective manifold $Y$, 
and take a resolution $\pi: \bar X \to X$ of the indeterminacy locus of $\phi$. 
Here we have the following commutative diagram$:$
\begin{equation*}
\xymatrix@C=40pt@R=30pt{
 & \bar X  \ar[d]_\pi \ar[rd]^{\bar{\phi}\ \ }  & \\ 
& X \ar@{-->}[r]^{\phi \ \ \ }  &  Y.\\   
}
\end{equation*}

(1) To prove the conclusion (1) (and also (3)) by using \cite[Corollary 2.11]{Hor07}, 
we will construct a foliation on $X$ (that is, an integrable subbundle of $T_X$) whose general leaf is rationally connected. 
We will show that the relative tangent bundle $T_{X/Y} \subset T_X$ 
(which is defined only on a Zariski open set of $X$) can be extended to a subbundle of $T_X$ on $X$. 
If it can be shown, it is not so difficult to check that this subbundle is integrable and its general leaf is rationally connected 
(that is, all the assumptions in \cite[Corollary 2.11]{Hor07} are satisfied).

Now we have the exact sequence of coherent sheaves 
\begin{align*}
0 \longrightarrow {\bar\phi}^{*} \Omega_Y \longrightarrow \Omega_{\bar X} 
\longrightarrow \Omega_{\bar X/Y}:=\Omega_{\bar X}/{\bar\phi}^{*} \Omega_Y \longrightarrow 0. 
\end{align*}
Then we obtain the injective sheaf morphism 
$0 \to \pi_*{\bar\phi}^{*} \Omega_Y \to \Omega_{X} $ 
by taking the pushforward. 
Here we used the formula $\pi_*\Omega_{\bar X}=\Omega_{X}$. 
By taking the dual, 
we obtain the exact sequence 
\begin{align}\label{eq-1}
0 \longrightarrow \mathcal{S}:= \Ker r \longrightarrow 
T_X \xrightarrow{\quad r \quad} \mathcal{Q}:= (\pi_*{\bar\phi}^{*} \Omega_Y)^{\vee}. 
\end{align}
We remark that the above sequence corresponds to 
the standard exact sequence of tangent bundles 
on a Zariski open set where $\phi$ is a smooth morphism. 

The morphism $r$ is generically surjective, 
and thus the reflexive sheaf $Q$ is also pseudo-effective. 
In particular, the first Chern class $c_{1}(\mathcal{Q})$ is also pseudo-effective. 
On the other hand, it follows that 
the image $Y$ of MRC fibrations has the pseudo-effective 
canonical bundle $K_Y$ from \cite{BDPP} and \cite{GHS03}. 
Further $\mathcal{Q}$ coincides
with the usual pullback of $T_{Y}$ on $X_0$. 
Here $X_0$ is the maximal Zariski open set where $\phi$ is a  morphism. 
Hence, by $\codim (X \setminus X_0) \geq 2$, 
it can be shown that 
$$
-c_{1}(\mathcal{Q})=c_{1}(\pi_*{\bar\phi}^{*}\Omega_Y)
=c_{1}(\pi_*{\bar\phi}^{*}K_Y)
$$
is pseudo-effective.

By the above argument, we can see that 
$\mathcal{Q}$ is a pseudo-effective reflexive sheaf 
with $c_1(\mathcal{Q})=0$, and thus we can conclude that $\mathcal{Q}$ 
is a numerically flat vector bundle on $X$ by Theorem \ref{theo-vect}.
By applying Lemma \ref{lemm-nonvanish} to $0\to \mathcal{Q}^\vee \to \Omega_{X}$
induced by (\ref{eq-1}), 
it can be seen that the sequence (\ref{eq-1}) is a bundle morphism on $X$. 
In particular, we can see that $\phi$ is smooth on $X_0$ 
(since the sequence (\ref{eq-1}) is not a bundle morphism on the non-smooth locus of $\phi$). 
The subbundle $\mathcal{S}$ defined by the kernel corresponds to 
the relative tangent bundle $T_{X/Y}$ defined on $X_0$. 
Hence $\mathcal{S}$  determines the foliation on $X$
since $T_{X/Y}$ is integrable on $X_0$ (for example, see \cite[subsection 2.2]{Mat18}). 
Further, its general leaf is rationally connected. 
Indeed, there exists a Zariski open set $Y_1$ in $Y$ such that 
$\phi: X_1:=\phi^{-1}(Y_1) \to Y_1$ is a proper morphism 
since $\phi: X \dashrightarrow Y$ is an almost holomorphic map 
(that is, general fibers are compact). 
A general leaf of $\mathcal{S}$ corresponds to a general fiber of $\phi$ 
by $\mathcal{S}=T_{X/Y}$ on $X_1$, 
and thus it is rationally connected. 
Therefore we can choose an MRC fibration to be holomorphic and smooth by \cite[Corollary 2.11]{Hor07}. 
We use the same notation $\phi: X \to Y$ for the smooth MRC fibration. 

\smallskip
(2) By (1), we have the standard exact sequence  
$$
0 \longrightarrow T_{X/Y} 
\longrightarrow T_X \longrightarrow \phi^* T_Y \longrightarrow 0,
$$
and also we have already checked that $\phi^* T_Y$ is pseudo-effective and $c_1(\phi^* T_Y)=0$. 
The pull-back $\phi^{*} T_Y $ is numerically flat by Theorem \ref{theo-vect}, 
and thus $T_Y$ is also numerically flat. 
The Beauville-Bogomolov decomposition (see \cite{Bea83}) asserts that  
there exists a finite \'etale cover $Y' \to Y$ such that 
$Y'$ is the product of hyperk\"ahler manifolds, Calabi-Yau manifolds, and abelian varieties. 
Let $Z$ be a component of $Y'$ of hyperk\"ahler manifolds or Calabi-Yau manifolds. 
We remark that $T_Z$ is also numerically flat. 
In general, numerically flat vector bundles are local systems (for example see \cite{DPS94}). 
Hence $T_Z$ should be a trivial vector bundle on $Z$ since $Z$ is simply connected and $T_Z$ is also numerically flat. 
This is a contradiction to the definition of hyperk\"ahler manifolds or Calabi-Yau manifolds. 
Hence the image $Y$ admits a finite \'etale cover $A\to Y$ by an abelian variety $A$. 

\smallskip

(4) We prove the conclusion (4). 
By considering the restriction of the standard exact sequence of the tangent bundle 
to a general fiber $F$, 
we obtain 
\begin{align*}
0 \longrightarrow T_{X/Y}|_F=T_F \longrightarrow T_{X}|_F 
\longrightarrow \phi^{*} T_Y|_F=N_{F/X}
=\mathcal{O}_F^{\oplus m} \longrightarrow 0. 
\end{align*}
When we consider the projective space bundle $f: \mathbb{P}(T_X) \to X$ and 
the non-nef locus $B \subset \mathbb{P}(T_X)$ of $\mathcal{O}_{\mathbb{P}(T_X)}(1)$, 
it can be seen that $f(B)$ is a proper subvariety of $X$ by pseudo-effectivity of $T_X$. 
By considering the commutative diagram 
\[
   \xymatrix{
    \mathbb{P}(T_X|_F)  \ar@{^{(}-_>}[r] \ar[d]^f & \mathbb{P}(T_X) \ar[d]^f   \\
    F \ar@{^{(}-_>}[r] & X 
   }
\]
we can see that the image of the non-nef locus of $\mathcal{O}_{\mathbb{P}(T_X|_F)}(1)$ 
is contained in $f(B \cap F)$. 
For a general fiber $F$, the image $f(B \cap F)$ is still a proper subvariety of $F$. 
Hence $T_{X}|_F$ is pseudo-effective. 
The surjective bundle morphism 
$$
\Lambda^{m+1} (T_{X}|_F)  \to T_F
$$
induced by the above exact sequence implies that $T_F$ is pseudo-effective.

\smallskip
We finally show that the MRC fibration $\phi : X \to Y$ is locally trivial if we further assume $X$ admits a positively curved singular hermitian metric. 
Under the assumption of such a metric, 
the exact sequence of the tangent bundle splits (that is, 
$T_X \cong T_{X/Y} \oplus \phi^{*}T_Y$)
by Theorem \ref{theo-split}. 
Then, by Ehrensmann's theorem (see also \cite[Lemma 3.19]{Hor07}), 
we can see that $\phi : X \to Y$ is locally trivial. 
\end{proof}

\begin{theo}\label{theo-posi-Ka}
Let $X$ be a compact K\"ahler manifold with pseudo-effective tangent bundle and 
$\phi: X \to Y:=\Alb(X)$ be its Albanese map. 
Then the Albanese map $\phi$ is a surjective smooth morphism and 
satisfies all the conclusions in Theorem \ref{re-theo-posi} except for $(3)$ and $(6)$
by replacing an abelian variety in $(2)$ with a compact complex torus. 
\end{theo}
\begin{proof}
In the proof of Theorem \ref{theo-posi}, 
the assumption of the projectivity was used only for the proof of (1) and (6). 
The other arguments except for (1) and (6) work 
even if we replace MRC fibrations with the Albanese map. 
Hence it is sufficient to prove that the Albanese map $\phi$ is a surjective smooth morphism. 
It is easy to check it. 
Indeed, for a basis $\{\eta_k\}_{k=1}^q$ of $H^{0}(X, \Omega_{X})$, 
it follows that any non-trivial linear combination of them is non-vanishing 
by Lemma \ref{lemm-nonvanish2}. 
This implies that $\phi$ is a surjective smooth morphism 
(for example see \cite{CP91}). 
\end{proof}

In \cite{DPS94}, it was proved that $X$ is a Fano manifold 
when $T_X$ is nef and $X$ is rationally connected. 
As an analog of this result, 
we suggest the following problem. 
We remark that the geometry of a general fiber $F$ in Theorem \ref{theo-posi} 
can be determined if the problem can be affirmatively solved. 

\begin{prob}\label{Fano}
If a projective manifold $X$ is rationally connected and has the pseudo-effective tangent bundle, then is the anti-canonical bundle $-K_X$ big?
\end{prob}

\section{Surfaces with pseudo-effective tangent bundle}\label{Sec4}

Toward the classification of surfaces with pseudo-effective tangent bundle, 
we study minimal ruled surfaces in subsection \ref{Sec4-1} 
and their blow-ups in subsection \ref{Sec4-2}, 
which provide interesting examples of positively curved vector bundles.

\subsection{On minimal ruled surfaces}\label{Sec4-1}
In this subsection, we consider a ruled surface $\phi: X \to C$ over a smooth curve $C$. 
When $T_X$ is pseudo-effective, 
the base $C$ should be either the projective line or an elliptic curve by Theorem \ref{theo-posi}. 
Conversely, it follows that any minimal ruled surfaces $\phi: X \to \mathbb{P}^1$ over $\mathbb{P}^1$  (that is, Hirzebruch surfaces)  
have the pseudo-effective tangent bundle from the following proposition. 
However, they do not have nef tangent bundle except for the case of $X=\mathbb{P}^1 \times \mathbb{P}^1$, 
since they have a curve with negative self-intersection.

\begin{prop}\label{prop-toric}
If $X$ is a projective toric manifold, then $T_X$ is generically globally generated. 
In particular, any Hirzebruch surfaces have pseudo-effective tangent bundle. 
\end{prop}
\begin{proof}
For a toric manifold $X$, we have an inclusion $(\mathbb{C}^*)^n \subset X$ as a Zariski open dense subset and an action $(\mathbb{C}^*)^n \curvearrowright X$. 
Consider a family of actions $(e^{i \theta},1, \cdots, 1)$. Differentiate it by $\theta$ at $\theta = 0$, we obtain a holomorphic vector field on $X$. Similarly, we can construct $n$ vector fields which generate $T_X|_{(\mathbb{C}^*)^n}$, 
and thus $T_X$ is generically globally generated.    
\end{proof}

Now we consider a ruled surface $\phi: X \to C$ over an elliptic curve $C$. 
Thanks to Theorem \ref{theo-posi}, 
we can see that the ruling $\phi: X \to C$ should be a smooth morphism when $X$ has the pseudo-effective tangent bundle. 
The minimal ruled surface $X$ over $C$ can be classified 
by \cite{Ati55}, \cite{Ati57}, and \cite{Suw69}: 
$X$ is isomorphic to $S_{n}$, $A_{0}$, $A_{-1}$, or a surface in $\mathcal{S}_0$. 
Here a surface $X$ in $\mathcal{S}_0$ means the projective space bundle $\mathbb{P}(\mathcal{O}_C \oplus L)$ for some $L \in \Pic^0(C)$ 
and $A_{0}$ (resp. $A_{-1}$) is the projective space bundle associated with 
a vector bundle of rank $2$ 
that is the non-split extension of $\mathcal{O}_C$ by $\mathcal{O}_C$ (resp. $\mathcal{O}_{C}(p)$), 
where $p$ is a point in $C$. 
It can be seen that $A_0$, $A_{-1}$, and surfaces in $\mathcal{S}_0$ have the nef tangent bundle by \cite{CP91}, 
and thus the remaining problem is the case of $X=S_n$. 
The ruled surface $S_{n}$ is the projective space bundle associated with   
the vector bundle $\mathcal{O}_C \oplus \mathcal{O}_C(np)$. 
Note that the tangent bundle of $S_{0}=\mathbb{P}^1 \times C$ is nef. 
By the above observation,  it is enough for our purpose to investigate $X=S_{n}$ in the case of $n \geq 1$.  
By the following proposition, 
we can see that   $S_{n}$ has the pseudo-effective tangent bundle (which is not nef), 
and further that it admits no positively curved singular hermitian metric.

\begin{prop}\label{prop-ell}
Let $\phi: X \to C$ be a minimal ruled surface over an elliptic curve $C$. 
Then we have:
\begin{itemize}
\item[(1)] The tangent bundle of $S_n$ is pseudo-effective, 
but it does not admit positively curved singular hermitian metrics when $n \geq 1$. 
\item[(2)] The tangent bundle of $S_0$, $A_0$, $A_{-1}$, and a surface in $\mathcal{S}_0$ is nef. 
\end{itemize}
\end{prop}

\begin{proof}
All the ruled surfaces with nef tangent bundle are classified in \cite{CP91}, 
which implies that the conclusion (2) holds and the tangent bundle of $S_n$ is not nef for $n \geq 1$.

From now on, let $X$ be the projective space bundle $S_n$ associated with  
the vector bundle $E_n:=\mathcal{O}_C \oplus \mathcal{O}_C(np)$. 
We first check the latter statement in the conclusion (1). 
If $X=S_n$ admits a positively curved singular hermitian metric, 
the exact sequence 
$$
0 \rightarrow T_{X/C} \rightarrow T_X \rightarrow \phi^{*} T_C    \rightarrow 0
$$
splits by Theorem \ref{theo-split}, and thus 
we have 
\begin{align}\label{eq4}
h^{0}(X, T_X)=h^{0}(X,T_{X/C}) + h^{0}(X,\phi^{*} T_C). 
\end{align}
On the other hand, we have $h^{0}(X, T_X)=n+1$ from \cite[Theorem 3]{Suw69}. 
Also we can easily check that 
$$
\phi_{*}(T_{X/C})=\phi_{*}(-K_X)=  \Sym^2(E_n) \otimes \det E_n^\vee.
$$
This implies that 
$$
h^{0}(X,T_{X/C})=h^{0}(C, \mathcal{O}_C(-np) \oplus \mathcal{O}_C \oplus \mathcal{O}_C(np) ) = n+1. 
$$
This is a contradiction to (\ref{eq4}). 

\smallskip

We will prove that $T_X$ is pseudo-effective. 
For this purpose, it is sufficient to prove that 
$\Sym ^m(T_X) \otimes \phi^{*}\mathcal{O}(2p)$ is generically globally generated 
for any $m \geq 0$. 
Our strategy is to observe a gluing condition of $X=S_n$ carefully to construct holomorphic sections that generate $\Sym ^m(T_X) \otimes \phi^{*}\mathcal{O}(2p)$ at general points.

Let $v$ be a local coordinate centered at $p$ and let $V \subset C$
 be a sufficiently small open neighborhood of $p$. 
Further, let $U$ be the open set $U:=C \setminus \{p\}$ and 
$u$ be the standard coordinate of the universal cover $\mathbb{C} \to C$. 
The ruled surface $X$ can be constructed 
by gluing  $(u, \zeta) \in U \times \mathbb{P}^1$ and $(v, \eta) \in V \times \mathbb{P}^1$ 
with the following identification: 
\begin{align}\label{gl}
\zeta = v^n \eta \quad \text{ and } \quad [u]=p+v, 
\end{align}
where $\zeta$ and $\eta$ are the inhomogeneous coordinates of $\mathbb{P}^1$.  

Let $\theta $ be a meromorphic section of $\Sym ^m(T_X)$ with pole along the fiber $\phi^{-1}(p)$ of $p$. 
Our strategy is as follows: 
We first look for a sufficient condition for the pole of $\theta$ being of order at most  $2$. 
Then we concretely  construct $\theta$ satisfying this condition, 
which can be regarded as a holomorphic section of $\Sym ^m(T_X) \otimes \phi^{*}\mathcal{O}(2p)$, 
and we show that such sections generate $\Sym ^m(T_X)\otimes \phi^{*}\mathcal{O}(2p)$ on a Zariski open set.

Now $\theta$ is a meromorphic section of $\Sym ^m(T_X)$ whose pole appears only along  the fiber $\phi^{-1}(p)$. 
Hence, by expanding $\theta$ on $U \times \mathbb{P}^1$, 
we have the following equality 
\begin{align}
\theta=\sum_{p=0}^{m} a_p(u, \zeta) \Big(\frac{\partial}{\partial \zeta}\Big)^{m-p} \Big(\frac{\partial}{\partial u}\Big)^{p} 
\text{ on } U \times \mathbb{P}^1. 
\end{align}
Here $a_p$ is a meromorphic function on $X$. 
The gluing condition (\ref{gl}) yields that 
\begin{align}
\frac{\partial}{\partial \zeta}=\frac{1}{v^n} \frac{\partial}{\partial \eta}
\quad \text{ and } \quad
\frac{\partial}{\partial u}=-n\frac{\eta}{v} \frac{\partial}{\partial \eta} + \frac{\partial}{\partial v}. 
\end{align}
Then we can obtain the following expansion of $\theta$ on  $V \times \mathbb{P}^1$ 
\begin{align}\label{ten2}
\theta =\sum_{\ell=0}^m 
\Big\{\sum_{p=\ell}^{m} d_{p, \ell} \,  a_p(v, \eta) \frac{\eta^{p-\ell}}{v^{n(m-p)+p-\ell}} \Big\}
\Big(\frac{\partial}{\partial \eta}\Big)^{m-\ell} \Big(\frac{\partial}{\partial v}\Big)^{\ell} 
\text{ on } V \times \mathbb{P}^1 
\end{align}
by an involved, but straightforward computation. 
Here $d_{p,\ell}$ is the non-zero constant defined by 
$d_{p,\ell}:=(-n)^{p-\ell} \binom{p}{p-\ell}$. 
The ruling $X \to C$ is locally trivial and sections of $\Sym^{p}(T_F)$ on a fiber $F$ are polynomials of degree (at most) $2p$. 
This implies that the meromorphic function  $a_{m-k}(u, \zeta)$ is a polynomial of degree $2k$ with respect to $\zeta$, 
and thus we can write $a_{m-k}$  as 
\begin{align}\label{ten}
a_{m-k}(v, \eta)=\sum_{q=0}^{2k} a_{m-k}^{(q)}(v)\, \zeta^q
=\sum_{q=0}^{2k} a_{m-k}^{(q)}(v)\, v^{nq} \eta^q \quad \text{ for any } 0 \leq k \leq m 
\end{align}
for some meromorphic function $a_{m-k}^{(q)}(v)$ on $C$ with pole only at  $p$. 
Here we used (\ref{gl}) again. 

We will find a sufficient condition for $a_{m-k}^{(q)}(v)$ for guaranteeing that 
the coefficients in (\ref{ten2}) have the pole  of order at most $2$. 
We remark that the section $\theta$ satisfying this condition determines the holomorphic section of $\Sym ^m(T_X) \otimes \phi^{*}\mathcal{O}(2p)$. 
By substituting (\ref{ten}) for (\ref{ten2}) and rearranging it concerning the powers of $\eta$, 
a sufficient and necessary condition can be obtained, 
but this method needs so complicated computation that we want to avoid to write down it. 
Here,  to improve our prospect,  we focus only on a sufficient condition 
by considering the restricted situation where $a_{m-k}^{(q)}=0$ for  $q \not= k$. 
In this situation, it is not so difficult to show that 
$\theta$ determines the holomorphic section of $\Sym ^m(T_X) \otimes \phi^{*}\mathcal{O}(2p)$
if $a_{m-q}^{(q)}$ satisfies that 
\begin{align}\label{ten3}
\sum_{p=0}^q d_{m-p, m-q} a_{m-p}^{(p)}(v) \frac{1}{v^{q-p}} \text{ has the pole of order $\leq 2$ at $p$ for any  $0 \leq q \leq m$}. 
\end{align}
For an explanation, we prepare the table where we write down them for $q=0, 1, 2$. 

\begin{table}[H]\label{table}
\begin{center}
\begin{tabular}{|l|c||r|r|} \hline
$q=q$ & $\text{coeff.of } (\partial/\partial \eta)^{q} (\partial/\partial v)^{m-q}$ & $\sum_{p=0}^q d_{m-p, m-q} a_{m-p}^{(p)} /v^{q-p}$
\\ \hline \hline
$q=0$ & $\text{coeff. of } (\partial/\partial \eta)^{0} (\partial/\partial v)^{m}$  & 
$d_{m,m} a_{m}^{(0)}$  \smallskip\\ \hline
$q=1$ & $\text{ coeff. of } (\partial/\partial \eta)^{1} (\partial/\partial v)^{m-1}$ & $d_{m,m-1} a_{m}^{(0)} /v+d_{m-1,m-1} a_{m-1}^{(1)}$  \smallskip \\ \hline
$q=2$ & $\text{coeff. of } (\partial/\partial \eta)^{2} (\partial/\partial v)^{m-2}$ & $d_{m,m-2} a_{m}^{(0)} /v^2+d_{m-1,m-2} a_{m-1}^{(1)} /v + d_{m-2,m-2} a_{m-2}^{(2)}$
\\ \hline
\end{tabular}
\end{center}
\end{table}

To construct meromorphic functions $a_{m-p}^{(p)}$ on $C$ satisfying (\ref{ten3}), 
for every $n \geq 2$, 
we take meromorphic functions $P_n$ on the elliptic curve $C$ such that 
$P_n$ has the pole only at  $p$ and 
its Laurent expansion at $p$ can be written as follows: 
$$
P_n(v)=\frac{1}{v^n}+\sum_{k \geq n+1}\frac{a_k}{v^k}. 
$$
Note that we can easily find them by using Weierstrass's elliptic functions and their differential. 

We first put $a_m^{(0)}:=P_2/d_{m, m}$. 
Then the second line from the top in the table satisfies  (\ref{ten3}) 
(that is, it has the pole of order  at most  $2$) 
if we define $a_{m-1}^{(1)}$ by $a_{m-1}^{(1)}:=-d_{m,m-1} /d_{m-1, m-1} P_3$. 
By the same way, the third line also satisfies  (\ref{ten3}) 
if we  define $a_{m-2}^{(2)}$ by an appropriate linear combination of $P_3$ and $P_4$. 
By repeating this process, 
we can construct meromorphic functions $a_{m-p}^{(p)}$ on $C$ satisfying (\ref{ten3}) 
by a linear combination of $\{P_3\}_{k=3}^{p+2}$. 
We denote by $\theta_0$ the holomorphic section of $\Sym ^m(T_X) \otimes \phi^{*}\mathcal{O}(2p)$ obtained from the above construction. 
The section $\theta_0$ generates the vector $(\partial/\partial \eta)^{0} (\partial/\partial v)^{m}$ on a Zariski open set, 
since $a_m^{(0)}=P_2/d_{m, m}$ is non-zero. 

Now we put $a_m^{(0)}:=0$ and $a_{m-1}^{(1)}:=P_2/d_{m-1, m-1} $, 
so that the first and the second line in the table have pole of  order at most $2$. 
Then, by the same argument as above, we can construct meromorphic functions $a_{m-p}^{(p)}$ 
satisfying (\ref{ten3})  by defining them  by an appropriate linear combination of $\{P_k\}_{k=3}^{p+2}$ 
(for example $a_{m-2}^{(2)}:=- d_{m-1, m-2}/d_{m-2, m-2} P_3$). 
We denote by $\theta_1$ the obtained holomorphic section of $\Sym ^m(T_X) \otimes \phi^{*}\mathcal{O}(2p)$. 
By the construction, the function $a_m^{(0)}$ is zero and $a_{m-1}^{(1)}$  is non-zero. 
Hence it follows the sections $\theta_0$ and $\theta_1$ generate the vectors
$(\partial/\partial \eta)^{0} (\partial/\partial v)^{m}$ and 
$(\partial/\partial \eta)^{1} (\partial/\partial v)^{m-1}$ on a Zariski open set.

By repeating this process, we can construct holomorphic sections $\{\theta_{p}\}_{p=0}^m$ of 
$\Sym ^m(T_X) \otimes \phi^{*}\mathcal{O}(2p)$ generating $\Sym ^m(T_X) \otimes \phi^{*}\mathcal{O}(2p)$ on a Zariski open set. 
\end{proof}

In the rest of this subsection, we suggest the following problem to investigate a gap between almost nefness and pseudo-effectivity of vector bundles. 

\begin{prob}\label{prob-ps}
We consider an exact sequence of vector bundles
\begin{align*}
0 \longrightarrow S \longrightarrow E
\longrightarrow Q \longrightarrow 0. 
\end{align*}
When $S$ and $Q$ are pseudo-effective, 
then is $E$ pseudo-effective?

\end{prob}

\begin{rem}\label{rem-almost}
When $S$ and $Q$ are nef, its extension $E$ is also nef (see \cite[Proposition 1.15]{DPS94}). 
Hence we can easily show that $E$ is almost nef if $S$ and $Q$ are almost nef. 
In particular, it can be shown that $\mathcal{O}_{E}(1)$ is pseudo-effective by \cite{BDPP}, 
but we do not know whether or not $E$ itself is pseudo-effective. 
The difficulty is to show that the image of the non-nef locus 
$\mathcal{O}_{E}(1)$ to $X$ is properly contained in $X$. 
If Problem \ref{prob-ps} can be affirmatively solved, 
the pseudo-effectivity of the tangent bundle of $X=S_n$ is easily obtained, 
by applying it to the standard exact sequence of the tangent bundle. 
In fact, we tried some methods in \cite{Suw69}, \cite{DPS94}, and \cite{Har70} to solve Problem \ref{prob-ps}, 
but it did not succeed. 
This problem seems to be subtle since we do not know whether there is a gap between almost nefness and pseudo-effectivity. 
\end{rem}

\subsection{On rational surfaces}\label{Sec4-2}
By the results in Subsection \ref{Sec4-1}, 
it is enough for the classification of the surfaces 
to determine when the blow-up of the Hirzebruch surface  
has pseudo-effective tangent bundle. 
However, it seems to be a too hard problem to classify all the  blow-ups completely 
since $X$ delicately depends on the position and the number of blow-up points. 
In this subsection, we study only blow-ups along {\textit{general points}}. 
The complete classification can not be achieved even in this case, 
but we obtain an interesting relation between positivity of tangent bundle and the geometry of Hirzebruch surfaces. 
The following proposition gives the requirement for the blow-up having pseudo-effective tangent bundle. 

\begin{prop}\label{prop-blow}
Let $\phi: \mathbb{F}_n \to \mathbb{P}^1$ be the Hirzebruch surface and let $\pi: X \to \mathbb{F}_n$ be the blow-up along the set $\Sigma$ of general points on $\mathbb{F}_n$. 
Then we have:
\begin{itemize}
\item[(1)] If the tangent bundle $T_X$ of $X$ is generically globally generated, 
then $\sharp \Sigma \leq 2$. 
\item[(2)] If the tangent bundle $T_X$ of $X$ is pseudo-effective, 
then $\sharp \Sigma \leq 4$. 
\end{itemize}
\end{prop}
\begin{rem}\label{rem-blow}
The interesting point here is that 
the conclusion of $\sharp \Sigma \leq 2$ in (1) is optimal, 
and further the generic global generation and pseudo-effectivity 
differently behave for $\sharp \Sigma$. 
Indeed, it follows that the tangent bundle $T_X$ 
in the case of $\sharp \Sigma \leq 3$ is pseudo-effective, 
but not generically globally generated from Proposition \ref{prop-rem}. 
\end{rem}
\begin{proof}
    (1) Fix a holomorphic vector field $\xi$ on $X$. We shall define a holomorphic vector field $\theta_\xi$ on $\mathbb{P}^1$ as follows. Let $t$ be a local holomorphic coordinate on $U \subset \mathbb{P}^1$. 
By pulling back $dt$, we obtain 
a holomorphic $1$-form $\pi^* \phi^* dt$ on $\widetilde{U} := (\pi \circ \phi)^{-1}(U)$. Then
$\langle \xi, \pi^* \phi^* dt \rangle$
is a holomorphic function on $\widetilde{U}$. Thus it is constant along each fiber and defines a holomorphic function on $U$.
Now we define the holomorphic vector field $\theta_\xi$ on $\mathbb{P}^1$ to be 
$$
\theta_\xi:=\langle \theta_\xi, dt\rangle \, \frac{\partial}{\partial t} 
\quad \text{ and } \quad \langle \theta_\xi, dt\rangle := \langle \xi, \pi^*\phi^*dt\rangle.
$$
 Since we assumed that $T_X$ is generically globally generated, we can choose $\xi$ with $\theta_\xi \not \equiv 0$ on $\mathbb{P}^1$.
    
We claim that $\theta_\xi$ has zeros on the set $\phi(\Sigma) \subset \mathbb{P}^1$. 
To prove the claim, we take a local coordinate $(t,s)$ on $\mathbb{F}_n$ centered at a point in $\Sigma$ such that $t$ is the pull-back of a local coordinate on $\mathbb{P}^1$. 
If we put $v := t/s$, then $(v,s)$ is a coordinate on $X$. 
Then we have
$$
\langle\xi, \pi^* \phi^* dt \rangle = \langle \xi, d(vs) \rangle = \langle \xi, s dv + v ds \rangle.
$$
The last term vanishes at $(v,s) =(0,0)$, and thus $\langle \theta_\xi, dt \rangle = 0$ at $t=0$. 
This shows the claim.
    
In the case of $\sharp \Sigma \geq 3$, 
the vector field $\theta_\xi$ has at least three zeros on $\mathbb{P}^1$.
It contradicts to the fact  of $\deg T_{\mathbb{P}^1} = 2$, thus we have $\sharp \Sigma \leq 2$. 

\smallskip
(2) 
    Since $T_X$ is pseudo-effective, we can choose an ample line bundle $A$ and a sequence of positively curved singular hermitian metrics $h_m$ on $(\Sym^m T_X) \otimes A$. Fix a smooth hermitian metric $h_A$ on $A$ with positive curvature.
    Then $h_m \otimes h_A^{-1}$ is a (possibly not positively curved) singular hermitian metric on $\Sym^m T_X$. Define a singular hermitian metric $g_m$ on $\pi^* \phi^* T_{\mathbb{P}^1}$ by the $m$-th root of the quotient metric of $h_m \otimes h_A^{-1}$ induced by the morphism $\Sym^m T_X \to (\pi^* \phi^* T_{\mathbb{P}^1})^{\otimes m}$. Since $(h_m \otimes h_A^{-1}) \otimes h_A$ is positively curved, 
    the metric $g_m^m \otimes h_A$ is also positively curved. 
 The curvature current $\sqrt{-1}\Theta_{g_m}$ of $g_m$ satisfies that 
    $$\sqrt{-1}\Theta_{g_m} \geq -\frac{1}{m} \omega_A.$$
    Then by taking a subsequence (if necessary), 
    we can assume that $\sqrt{-1}\Theta_{g_m}$ weakly converges to a positive current $T \in c_1(\pi^*\phi^* T_{\mathbb{P}^1})$. 
    By the argument similar to  (1),  we obtain a $d$-closed positive $(1,1)$-current $S$ in $c_1(T_{\mathbb{P}^1})$ such that $T =  \phi^* \pi^*S$.
    Hence we have 
    $$
    \sqrt{-1}\Theta_{g_m} \to \pi^* \phi^* S = T \in c_1(\phi^*\pi^* T_{\mathbb{P}^1}). 
    $$

    We take a point $p \in \Sigma$ and put $p_0:=\phi(p)$. 
    We claim that the following bound of the Lelong number
    \begin{equation}
    \nu(S, p_0) \geq \frac{1}{2}.\label{claim_lelong}
    \end{equation}
    We fix a local coordinate $t$ near $p_0 \in \mathbb{P}^1$.  
    Let $(t,s)$ be a coordinate on $\mathbb{F}_n$ centered at $p$.   
    As before, by  putting $v = t/s$, $(v,s)$ is a coordinate on $X$. Let $p' \in X$ be a point defined by $(v,s) = (0,0)$.
    Let $C$ be a (local) holomorphic curve on $X$ defined by $\{v=s\}$. We will denote $\overline{C} := \pi(C)$. The defining equation of $\overline{C}$ is $\{ t/s = s \} = \{ t = s^2\}$. Then we have 
\begin{equation}
\nu(S, p_0) = \frac{1}{2} \nu(\phi^* S|_{\overline{C}}, p).\label{eqn:Lelong-num}
\end{equation}
    Indeed, the function $\phi^* \gamma$ is a local potential of $\phi^* S$ for a local potential $\gamma$ of $S$. 
    Note that $s$ is a local coordinate on $\overline{C}$ while $t= s^2$ is a local coordinate on $\mathbb{P}^1$. We can calculate each Lelong number by the formula
    $$ \nu(S, p_0) = \liminf_{t \to 0} \frac{\gamma(t)}{\log |t|},$$
    and thus
    $$ \nu(\phi^* S|_{\overline{C}}, p) = \liminf_{s \to 0 } \frac{\phi^*\gamma(s^2,s)}{\log |s|} = \liminf_{s \to 0 } \frac{\gamma(s^2)}{\log |s|} = 2 \nu(S, p_0). $$
    This proves (\ref{eqn:Lelong-num}).
    Since the Lelong number will increase after taking restriction, we have 
    $$\nu(\phi^* S|_{\overline{C}}, p)=\nu(T|_C, p') \geq \nu(T, p').$$
    Lelong numbers will also increase after taking a weak limit of currents, thus we obtain 
    $$\nu(T, p')\geq \limsup_{m \to +\infty}\nu(\sqrt{-1}\Theta_{g_m}, p').$$
    The local weight of $g_m$ is written as 
    $$\frac{1}{2m} \log |(\pi^*\phi^* (dt))^m|^2_{h_m^{-1} \otimes h_A}. $$ 
    Since $t = vs$ on $X$, we can calculate as follows:
    \begin{equation}
    |(\pi^*\phi^* (dt))^m|^2_{h_m^{-1} \otimes h_A} = |(vds + sdv)^m|^2_{h_m^{-1} \otimes h_A}.\label{eqn:Lelong-vector}
    \end{equation}
    Since $h_m^{-1}$ is negatively curved and $h_A$ is smooth, it follows that 
    $$|\cdot|^2_{h_m^{-1} \otimes h_A} \leq C_0 |\cdot|^2_{h_{\rm sm}} $$
    for a smooth hermitian metric $h_{\rm sm}$ and some constant $C_{0} >0$ (both depending on $m$). Then the right-hand side of (\ref{eqn:Lelong-vector}) is bounded as
    \begin{align*}
    &\leq C_0 |(vds + sdv)^m|^{2}_{h_{\rm sm}}\\
    &\leq C_0 |(v,s)|^{2m}.
    \end{align*}
    Thus, the Lelong number of $\sqrt{-1}\Theta_{g_m}$ is bounded as
    $$ \nu(\sqrt{-1}\Theta_{g_m}, p') \geq \frac{1}{2m}\liminf_{(v,s) \to 0}\frac{C_{0} |(v,s)|^{2m}}{\log|(v,s)|} = 1.$$
    This proves (\ref{claim_lelong}). Since $\deg T_{\mathbb{P}^1} = 2$, there must be at most four points where $S$ has the Lelong number greater than or equal to $1/2$. Therefore $\sharp \Sigma \leq 4$.
\end{proof}

We finally prove Proposition \ref{prop-rem} by applying the following lemma. 
The lemma is useful when we compare a vector field on a given manifold 
with its blow-up.

\begin{lemm}\label{lem-bup}
Let $\pi: Y \to \mathbb{C}^2$ be the blow-up at  $(\alpha,\beta) \in \mathbb{C}^2$ with the exceptional divisor $E$, 
and let $(x,y)$ be the standard coordinate of  $\mathbb{C}^2$.
We consider a  holomorphic section $\theta$ of $\Sym^{m}T_{\mathbb{C}^2}$ and its expansion 
$$
\theta = \sum_{k=0}^m f_k(x,y) \Big(\frac{\partial}{\partial x}\Big)^{k} \Big(\frac{\partial}{\partial y}\Big)^{m-k}. 
$$
Then the pull-back $(\pi|_{Y \setminus E})^{*} \theta$ by the isomorphism $\pi|_{Y \setminus E}$ on $Y \setminus E$ 
can be extended to the holomorphic section of $\Sym^{m}T_{Y}$ if and only if 
$$
\sum_{k=0}^m f_k( s+\alpha ,st+\beta) 
\Big(\frac{\partial}{\partial s} -  \frac{t}{s}\frac{\partial}{\partial t} \Big)^{k} 
\Big(\frac{1}{s}\frac{\partial}{\partial t}\Big)^{m-k}. 
$$
is holomorphic  with respect to $(s,t) \in \mathbb{C}^2$.
\end{lemm}
\begin{proof}
We first remark that any holomorphic section $\xi$ of $\Sym^{m}T_{Y}$ determines the section $\theta_{\xi}$ of $\Sym^{m}T_{{\mathbb{C}}^2 }$. 
Indeed, a given section $\xi$ induces the section $\theta_{\xi}$ of $\Sym^{m}T_{{\mathbb{C}}^2 }$ on $\mathbb{C}^2 \setminus \{(\alpha, \beta)\}$ 
via the isomorphism $\pi|_{Y \setminus E}$, 
which can be extended on $\mathbb{C}^2$ 
since the blow-up center has codimension two.

We consider the descriptions: 
\begin{align*}
Y &= \{ (x,y,[z  : w]) \in \mathbb{C}^2\times \mathbb{P}^1 \,|\,  (x-\alpha)w = (y-\beta)z \},  \\
E &= \{ (\alpha,\beta ,[z : w]) \,|\,  [z  : w] \in  \mathbb{P}^1 \}. 
\end{align*}
and put the Zariski open set  $Y' := Y \cap \{ w \neq 0 \}$.  
The following map $r$ gives a coordinate of $Y'$ and 
$\pi |_{Y'}$ can be written as follows: 
$$
\begin{array}{ccccccc}
r : \mathbb{C}^2                 & \rightarrow& Y'       & &\pi |_{Y'}  :   Y'  &\rightarrow & \mathbb{C}^2        \\
(s,t)                   &    \mapsto  & (s+\alpha,st+\beta,[1 : t]) &  &   (x,y,[z:w])&  \mapsto &(x,y)
\end{array}
$$
If $(\pi \circ r  )^{*}\theta$ is holomorphic on $\mathbb{C}^2  $, then
$(\pi|_{Y \setminus E})^{*} \theta$ can be extended to the holomorphic section of $\Sym^{m}T_{Y}$. 
Indeed, in this case, the section
$(\pi|_{Y \setminus E})^{*} \theta$ can be extended to the holomorphic section 
of  $\Sym^{m}T_{Y'  }$. 
Hence it can also be extended on $ Y $ since the codimension of $E \cap \{ w  = 0\}$ is two.

By calculation, we obtain
$$
(\pi \circ r )^{*}\theta = \sum_{k=0}^m f_k( s+\alpha ,st+\beta) 
\Big(\frac{\partial}{\partial s} -  \frac{t}{s}\frac{\partial}{\partial t} \Big)^{k} 
\Big(\frac{1}{s}\frac{\partial}{\partial t}\Big)^{m-k}. 
$$
Hence $(\pi \circ r )^{*}\theta$  is holomorphic on $\mathbb{C}^2$
if and only if the right hand side is holomorphic in $(s,t) \in \mathbb{C}^2$, which completes  the proof.
\end{proof}

\begin{prop}\label{prop-rem}
We have: 
\begin{itemize}
\item[(1)] The blow-up of the Hirzebruch surface $\mathbb{F}_n$ along general one or two points has 
the generically globally generated tangent bundle. 
\item[(2)] The blow-up of the Hirzebruch surface $\mathbb{F}_n$ along general three points has 
the pseudo-effective tangent bundle. 
\end{itemize}
\end{prop}

Because the general case is tedious, we first show Proposition \ref{prop-rem} in the simplest case $n=0$.

\begin{proof}[Proof of  $(1)$ for $\mathbb{F}_0$]

In general, for a birational morphism $f: Y \to Z$ between projective manifolds, 
we have the natural inclusion $f_* T_Y \subset T_Z$.  
Since the natural inclusion is of course generically isomorphism, 
$T_Z$ is generically globally generated if the tangent bundle $T_Y$ is so. 
Therefore it is sufficient for the proof of (1) to treat  only the blow-up $\pi: X \to \mathbb{F}_0$ along  general two points $p_1$, $p_2$. 

We take a Zariski open set $ \mathbb{C} \times \mathbb{C} = W_0 \subset \mathbb{F}_0$ with the local coordinate $(x,y)$. 
We may assume that $p_1= (0,0) $ and $p_2= (1,1) $ 
by using the action of the automorphism group of $\mathbb{F}_0$. 
We define the set of holomorphic vector fields on $W_0$
  \begin{flalign*}
 \mathcal{T} := 
 \Big\{  \sum_{k=0}^{2} a_k x^k \frac{\partial }{\partial x}
 + \sum_{l=0}^{2} b_l y^l   \frac{\partial }{\partial y}  \, \Big| \, a_k , b_{l} \in \mathbb{C}\Big\}.
  \end{flalign*}
We remark that any $\theta \in \mathcal{T}$ can be extended to a global holomorphic section of $\textit  T_{\mathbb{F}_0}$.
From Lemma \ref{lem-bup}, 
for a holomorphic vector field 
$$
\theta   := a(x){\partial }/{\partial x} +  b(y){\partial }/{\partial y}  \in \mathcal{T} 
$$
it follows that 
$\theta$ can be lifted to the holomorphic section  of $T_{Y}$
if and only if 
$$
\text{
$\frac{1}{s} \big( -a(s+\alpha)t+ b(st+\beta) \big)$ is holomorphic with respect to $(s,t) $
}
$$
for $(\alpha, \beta)=(0,0)$ and $(\alpha, \beta)=(1,1)$. 
We choose $\theta_1$ and $\theta_2$ in $\mathcal{T}$ as follows: 

\begin{equation*}
\theta _1 = (x^2-x)\frac{\partial }{\partial x}  \text{ and }  \theta _2 = (y^2-y)\frac{\partial }{\partial y}. 
\end{equation*}

Then we can easily see that $\pi^{*}\theta _1$ and $\pi^{*}\theta _2$ can be extended to the global holomorphic sections of $T_{X}$.
For a point $q=(x,y) \in W_0$ such that $x \neq 0,1 $ and $y \neq 0,1$, 
the vectors $\theta _1(q)$ and $\theta _2(q)$ at $q$ give a basis of $T_{W_0 , q}$. 
Therefore $T_{X}$ is generically globally generated.
\end{proof}

\begin{proof}[Proof of $(2)$ for $\mathbb{F}_0$]
We use the same notations as in the proof of (1).
Let $\pi : X \rightarrow \mathbb{F}_0$ be a blow-up of $\mathbb{F}_0$ along general three points $p_1,p_2,p_3$.  
Our goal in this proof is to show that $\Sym^{2} (T_{X})$ is generically globally generated.
Since $p_1,p_2,p_3$ are in general position, we may assume 
$p_1,p_2,p_3 \in W_0$, $p_1= (0,0) $, $p_2= (1,1) $, and $p_3= (-1,-1) $ 
by the action of the automorphism group of $\mathbb{F}_0$.

 We define $ \mathcal{T}$ by 
  \begin{flalign*}
 \mathcal{T} := 
 \left\{  \sum_{k=0}^{4} a_k x^k \Big( \frac{\partial }{\partial x} \Big)^{2}
 +  \sum_{0 \le k,l \le 2} b_{kl} x^{k}y^{l} \frac{\partial }{\partial x}  \frac{\partial }{\partial y}  
 + \sum_{k=0}^{4} c_k y^k  \Big( \frac{\partial }{\partial y}  \Big)^{2} 
\, \Big| \, a_k , b_{kl}, c_{k} \in \mathbb{C}\right\}.
  \end{flalign*}
 It is easy to show that any $\theta \in \mathcal{T}$ can be extended to a holomorphic global section of $\Sym^{2}  T_{\mathbb{F}_0}$.
 
By Lemma \ref{lem-bup}, we can see that, 
for a holomorphic section 
$$\theta = a(x)\Big( \frac{\partial }{\partial x}   \Big)^{2}+  b(x,y)\frac{\partial }{\partial x} \frac{\partial }{\partial y}  + c(y)  \Big( \frac{\partial }{\partial y}  \Big)^{2} \in \mathcal{T} , 
$$
the section $\theta$ can be lifted to the section of $\Sym^{2}  T_{\mathbb{F}_0}$ 
if and only if  the followings are holomorphic  with respect to $(s,t) \in \mathbb{C} \times  \mathbb{C}  $: 
 \begin{align*}
  &\frac{1}{s} \big( -2a(s+\alpha,st+\beta)t+ b(s+\alpha,st+\beta)\big),  \\
  &\frac{1}{s^2} \big(a(s+\alpha,st+\beta)t^2- b(s+\alpha,st+\beta)t + c(s+\alpha,st+\beta)\big), 
 \end{align*}
 for $(\alpha, \beta)=(0,0), (1,1), (-1,-1)$. 
 
 Here we put 
 \begin{align*}
& \theta_1 = y^2(x^2-1)\frac{\partial }{\partial x}  \frac{\partial }{\partial y}  + y^2(y^2-1)  \Big( \frac{\partial }{\partial y}  \Big)^{2}, \\
& \theta_2 =  x^2(x^2-1)\Big( \frac{\partial }{\partial x}   \Big)^{2} + x^2(y^2-1)\frac{\partial }{\partial x}  \frac{\partial }{\partial y},\\
&\theta_3=(x-y)^2\frac{\partial }{\partial x}  \frac{\partial }{\partial y}.
 \end{align*}
Then we can easily show that $\pi^{*}\theta_1$, $\pi^{*}\theta_2$, and $ \pi^{*}\theta_3$ can be 
extended to global holomorphic  sections of $\Sym^{2}  T_{X}$. 
For a general point $q \in W_0$, it is easy to see that 
$\theta_1(q)$, $\theta_2(q)$, and $\theta_3(q)$ give a basis of $\Sym^{2}  T_{W_0 , q}$. 
Therefore $\Sym^{2} T_{ X}$ is generically globally generated.
\end{proof}

As a preliminary of the proof for $\mathbb{F}_n$, 
we prove the following claim. 
We regard the Hirzebruch surface $\mathbb{F}_n$ for $n \geq 1$ as the hypersurface in $\mathbb{P}^1\times \mathbb{P}^2$ 
$$
\mathbb{F}_n = \{ ([X_1:X_2],[Y_0 : Y_1 : Y_2] ) \in \mathbb{P}^1\times \mathbb{P}^2 \,|\, Y_{1}X_{2}^{n}= Y_{2}X_{1}^{n} \}. 
$$
We set $ U = \{ Y_1 \neq 0 \text{ or } Y_2 \neq 0  \}$. 
We first observe  the automorphism group of $\mathbb{F}_n$ 
so that general three points move to specific points, 
which makes our computation not so hard. 
\begin{claim}
\label{claim-Hil}
General three points $p_1,p_2,p_3 \in U$ move to 
$([1:0],[1:1:0]),([1:1],[1:1:1]),([1:-1],[1:1:(-1)^n])$ by the action of the automorphism group of $\mathbb{F}_n $.
\end{claim}
\begin{proof}
Let $S,T$ be variables and  $P_n$ be a  vector subspace of homogeneous polynomials of degree $n$ in $\mathbb{C}[S,T]$.
The linear group $ \GL(2,\mathbb{C})$  acts on  $P_n$ as follows:
For any $\left( \begin{matrix}
a & b \\
c & d 
\end{matrix} \right) \in  \GL(2,\mathbb{C})$ and any $\sum_{k=0}^{n }a_k S^{k}T^{n-k} \in P_n$, we define the action by
$$
\left( \begin{matrix}
a & b \\
c & d 
\end{matrix} \right)  \bullet \Big( \sum_{k=0}^{n }a_k S^{k}T^{n-k}  \Big) :=
\sum_{k=0}^{n }a_k (aS+bT)^{k}(cS+dT)^{n-k}.
$$
This induces the semidirect product $ G_n :=P_n \rtimes \GL(2,\mathbb{C})$.

For any 
$
g=( \sum_{k=0}^{n }a_k S^{k}T^{n-k}, 
\left( \begin{matrix}
a & b \\
c & d 
\end{matrix} \right) 
) \in G_n, 
$
we define the action of $\mathbb{F}_n$ as follows: 
For any $q=([X_1:X_2],[Y_0 : Y_1 : Y_2] ) \in \mathbb{F}_n$, we define  $g(q)$ by  
$$
([aX_1 + bX_2 : cX_1 + d X_2], [ Y_{0}X_{1}^n + Y_1 \sum_{k=0}^{n}a_{k} X_{1}^{k} X_{2}^{n-k}  : Y_1(aX_1 + bX_2)^n : Y_1(cX_1 + d X_2)^n]),
$$
 if $X_1 \neq 0$ and by 
$$
([aX_1 + bX_2 : cX_1 + d X_2],
[ Y_{0}X_{2}^n + Y_2 \sum_{k=0}^{n}a_{k} X_{1}^{k} X_{2}^{n-k}  : Y_2(aX_1 + bX_2)^n : Y_2(cX_1 + d X_2)^n]) \\
$$
if $X_2 \neq 0$ (see \cite[Theorem 4.10]{DI09} or \cite[Section 6.1]{Bla12}).

Note that the ruling $\phi : \mathbb{F}_n  \rightarrow \mathbb{P}^1$ 
coincides with the first projection.
We may assume that $p_1$, $p_2$ and $p_3$ are in $U$ 
and also that the images of them in $\mathbb{P}^1$ are different from each other. 
By the action of $g=(0, 
\left( \begin{matrix}
a & b \\
c & d 
\end{matrix} \right) 
)  $,
we obtain 
$$
\phi (g(p_1))=[1:0], \quad \phi (g(p_2))=[1:1],  \quad \phi (g(p_3))=[1:-1]
$$
if we properly choose $g$.
Therefore we may assume 
$$
p_1 =([1:0] ,[x_1 : y_1:0]) , p_2 =([1:1] ,[x_2 : y_2 : y_2]) , p_3 =([1:-1] ,[x_3 : y_3 : (-1)^{n} y_3]).
$$
It follows that $y_k \neq 0$ for $k = 1,2,3$ 
since we have $ g \cdot U \subset U$ for any $g\in G_n$.

In the case of $n =1$ 
we put
$$
a=\frac{x_1}{y_1} - \frac{x_2}{2y_2} -\frac{x_3}{2y_3}, \quad 
a_0= - \frac{x_2}{2y_2} -\frac{x_3}{2y_3}, \quad 
a_1=\frac{x_1}{y_1} - \frac{x_2}{y_2}. 
$$
Then $p_{1}$, $p_{2}$, $p_{3}$ respectively move to $([1:0],[1:1:0])$, $ ([1:1],[1:1:1])$, $([1:-1],[1:1:(-1)^n])$ 
by the action of $(a_{0}S + a_{1}T , 
\left( \begin{matrix}
a & 0 \\
0 & a 
\end{matrix} \right) 
) \in G_1,
$
since we may assume ${x_1}/{y_1} - {x_2}/{2y_2} -{x_3}/{2y_3} \neq 0 $  since $p_1,p_2,p_3$ are general points.

In the case of $n \geq 2$, 
we put $m = 2\lfloor n/2 \rfloor$, 
$$
a_0=\frac{x_1 - y_1}{y_1}, \quad 
a_1=-\frac{x_2-y_2}{2y_2} + \frac{x_3 + y_3}{2y_3}, \quad 
a_m=- \frac{x_1 - y_1}{y_1} - \frac{x_2 - y_2}{2y_2} -\frac{x_3 + y_3}{2y_3}, 
$$
and $a_k = 0$ for $k \neq 0,1,m$.
Then $p_{1}$, $p_{2}$, $p_{3}$ respectively move to $([1:0],[1:1:0])$, $([1:1],[1:1:1])$, $([1:-1],[1:1:(-1)^n])$ 
by the action of $( \sum_{k=0}^{n}a_{k}S^{k}T^{n-k} , 
\left( \begin{matrix}
1 & 0 \\
0 & 1 
\end{matrix} \right) 
) \in G_n
$.
\end{proof}

\begin{proof}[Proof of $(1)$ for $\mathbb{F}_n$]
We define the Zariski open sets $W_k \cong \mathbb{C} \times \mathbb{C} $ in $\mathbb{F}_n$ for $k=1,2,3$ as follows:
$$
 \begin{array}{cccccc}
 i_1 :  W_1 &\rightarrow & \mathbb{F}_n 
 &  i_2 :  W_2 &\rightarrow & \mathbb{F}_n \\

   (x,y )&\mapsto  &  ([1:x], [1 : y :x^{n}y] ),
 & (u,v )&\mapsto  &  ([1:u],[v: 1 : u^{n} ]),

 \end{array}
 $$
 $$
  \begin{array}{ccc}
   i_3 :  W_3 &\rightarrow &  \mathbb{F}_n \\
  (\zeta,\eta )&\mapsto  &  ([\zeta : \eta] ,[1:\zeta^{n}\eta :\eta]). 
 \end{array}
 $$

We take $\theta = a(x,y)\partial/\partial x + b(x,y) \partial / \partial y \in H^0(W_1,T_{W_1})$.
The section $\theta$ extends to a holomorphic global section of $T_{\mathbb{F}_n}$ if and only if 
 $ \theta$ is holomorphic on $W_2$ and $W_3$,
since the codimension of $\mathbb{F}_n \setminus \cup_{k=1,2,3 } W_k$ is two.
A straightforward computation yields 
\begin{align*}
&\theta =a(u,1/v)\frac{\partial}{\partial u} -v^2b(u,1/v)\frac{\partial}{\partial v} 
\quad \text{ on $ W_1 \cap W_2$, }\\
&\theta =-\zeta^2a(1 /\zeta,\zeta^n \eta)\frac{\partial}{\partial \zeta} 
 + \Big( n \zeta \eta a(1 / \zeta,\zeta^n \eta) + \frac { b(1 / \zeta,\zeta^n \eta) }{\zeta^n}  \Big) \frac{\partial}{\partial \eta} \quad \text{ on $ W_1 \cap W_3$.}
\end{align*}
Hence it can be seen that 
the section $\theta$ can be extended  to a global holomorphic  section of $T_{\mathbb{F}_n}$
if and only if we define $a(x,y)$  and $b(x,y)$ to be 
\begin{align*}
a(x,y) = a_0 + a_1 x+a_2 x^2 \quad  \text{ and }\quad 
b(x,y) =  (b_0 - n a_2 x)y + b_1(x) y^2
\end{align*}
for some $a_0 , a_1  , a_2, b_0\in \mathbb{C}$ and for some $b_1(x) \in \mathbb{C}[x]$ with $ \deg(b_1) \le n$.
We define
$$
\mathcal{T} := \Big\{ (a_0 + a_1 x+a_2 x^2)\frac{\partial }{\partial x} +  (b_{0}y  - n a_{2} xy+ b_{1} y^2 + b_{2}x y^2)\frac{\partial }{\partial y} 
\,\Big|\,  a_0 , a_1  , a_2, b_0 ,b_1,b_2 \in \mathbb{C} \Big\}. 
$$
Then, by the above observation, it can be seen that 
any $\theta \in \mathcal{T}$ extends to a holomorphic global section of $T_{\mathbb{F}_n}$.

Let $\pi : X \rightarrow \mathbb{F}_n$ be the blow-up  of $\mathbb{F}_n$ along general two points $p_1,p_2$.
By Claim \ref{claim-Hil}, we may assume $p_1,p_2 \in W_1 $,
$p_1= (0,1) $ and $p_1= (1,1) $.
We choose $\theta_1$ and $\theta_2$ in $\mathcal{T}$ as follows: 
$$
\theta _1 = y(y-1)\frac{\partial }{\partial y} 
\quad  \text{ and }\quad 
\theta _2 = x(x-1)\frac{\partial }{\partial x} +nxy(y-1)\frac{\partial }{\partial y}. 
$$
By  Lemma \ref{lem-bup}, 
the sections $\theta _1$ and $\theta _2$ can be lifted  to holomorphic global sections of $T_{X}$.
For any point $q=(x,y) \in W_1$ such that $x \neq 0,1 $ and $y \neq 0,1$, $(\theta _1)_q$ and $(\theta _2)_q$ give a basis of $T_{W_1 , q}$. 
Therefore $T_{ X}$ is generically globally generated.
\end{proof}

\begin{proof}[Proof of  $(2)$ for $\mathbb{F}_n$]
Let $\pi : X \rightarrow \mathbb{F}_n$ be a blow-up of $\mathbb{F}_n$ along general three points $p_1,p_2,p_3$.  We show that 
$\Sym^{2} (T_{X})$ is generically globally generated.
By Claim \ref{claim-Hil}, we may assume 
$p_1,p_2,p_3 \in W_1$,
$p_1= (0,1) $, $p_2= (1,1) $,  and $p_3= (-1,-1) $. 

We take 
$$
\theta = a(x,y)\Big( \frac{\partial }{\partial x}   \Big)^{2}+  b(x,y)\frac{\partial }{\partial x}  \frac{\partial }{\partial y}  + c(x,y) \Big( \frac{\partial }{\partial y}  \Big)^{2}\in H^0(W_1,\Sym^{2}  T_{W_1}). 
$$ 
First we investigate the  condition when $\theta$ extends to a global holomorphic  section of $\Sym^{2}  T_{\mathbb{F}_n}$.
We have
$$
\theta =
a(u,1/v)\Big( \frac{\partial }{\partial u}  \Big)^{2}
-v^2b(u,1/v)\frac{\partial}{\partial u}\frac{\partial}{\partial v} 
+v^{4}c(u,1/v)\Big( \frac{\partial }{\partial v}  \Big)^{2}
\text{ on $ W_1 \cap W_2$ and, }
$$ 
 \begin{align*}
\theta & =
\zeta^{4}a(1/\zeta,\zeta^{n} \eta)\Big( \frac{\partial }{\partial \zeta}  \Big)^{2}
+\Big( -2n\zeta^{3}\eta a(1/\zeta,\zeta^{n} \eta) -\frac{1}{\zeta^{n-2}} b(\zeta,\zeta^{n} \eta) \Big) \frac{\partial}{\partial \zeta} \frac{\partial}{\partial \eta} \\
 &+\Big( n^{2}\zeta^{2}\eta^{2} a(1/\zeta,\zeta^{n} \eta)  +\frac{n \eta}{\zeta^{n-1}} b(1/\zeta,\zeta^{n} \eta)  
 +\frac{1}{\zeta^{2n}}c(1/\zeta,\zeta^{n} \eta) \Big) \Big( \frac{\partial }{\partial \eta}  \Big)^{2} \text{ on $ W_1 \cap W_3$.}
\end{align*}

In the case of $n=1$, the section 
$\theta$ extends to a  global holomorphic section of $\Sym^{2}  T_{\mathbb{F}_n}$ if we have 
\begin{itemize}
\item $a(x,y)=a_0 + a_{1}x+a_{2}x^2+a_{3}x^3+a_{4}x^4 $,
\item $b(x,y)= (b_{0}+b_{1}x+b_{2}x^{2}-2a_{4}x^{3} )y+(b_{3}+b_{4}x+b_{5}x^2 +b_{6}x^3)y^2$,
\item$ c(x,y)= (c_{0}-(a_{3} +b_{2})x+a_{4}x^2)y^2+(c_1+c_{2}x  -  b_{6} x^2)y^3
+(c_{3}+c_{4}x+c_{5}x^2+c_{6}x^3+c_{7}x^4)y^4$,
\end{itemize}
where all coefficients are constant.
Here we put 
  \begin{align*}
\bullet \ \  \theta_1 &=
 x(x^2 - 1)\Big( \frac{\partial }{\partial x}   \Big)^{2} 
+y\Big(-3x^2 + y(x^3 + x^2 + x - 1) + 1\Big)\frac{\partial }{\partial x}  \frac{\partial }{\partial y}  \\
&+y^2\Big(2x + y^2(x^2 + 1) - y(x + 1)^2\Big)\Big( \frac{\partial }{\partial y}  \Big)^{2},\\
\bullet \ \   \theta_2 &=  
 x^2(-x^2 + 1)\Big( \frac{\partial }{\partial x}   \Big)^{2} 
 +2x^2y(x - y)\frac{\partial }{\partial x} \frac{\partial }{\partial y}  
+x^2y^2(y^2 - 1)  \Big( \frac{\partial }{\partial y}  \Big)^{2},\\
\bullet  \ \  \theta_3&=
x(-x^2 + 1)\Big( \frac{\partial }{\partial x}   \Big)^{2} 
+ y\Big(3x^2 + y(-x^2 - 2x + 1) - 1\Big)\frac{\partial }{\partial x}  \frac{\partial }{\partial y} \\ 
&+y^2\Big(-2x + y^2(2x - 1) + 1\Big)\Big( \frac{\partial }{\partial y}  \Big)^{2}.
 \end{align*}
Then, by using Lemma \ref{lem-bup} again, 
the sections $\pi^{*}\theta_1$, $\pi^{*}\theta_2$ and $ \pi^{*}\theta_3$ 
extend to holomorphic global sections of $\Sym^{2}  T_{X}$.
For a general point $q \in W_1$, $\theta_1(q)$, $\theta_2(q)$, and $\theta_3(q)$ give basis of $\Sym^{2}  T_{W_1, q}$. 
Therefore $\Sym^{2} T_{ X}$ is generically globally generated.

In the case of $n \geq 2$, 
the section $\theta$ extends to a holomorphic global section of $\Sym^{2}  T_{\mathbb{F}_n}$ if
\begin{itemize}
\item $a(x,y)=a_0 + a_{1}x+a_{2}x^2+a_{3}x^3+a_{4}x^4 $,
\item $b(x,y)= (b_{0}+b_{1}x+b_{2}x^{2}-2na_{4}x^{3} )y+(b_{3}+b_{4}x+b_{5}x^2 +b_{6}x^3)y^2$,
\item$ c(x,y)= (c_{0} - (n^{2}a_{3}  + nb_{2})x+n^{2}a_{4}x^2)y^2+(c_1+c_{2}x + c_{3} x^2)y^3
+(c_{4}+c_{5}x+c_{6}x^2+c_{7}x^3+c_{8}x^4)y^4$,
\end{itemize}
where all coefficients are constant.
We put 
\begin{align*}
& \bullet \ \  \theta_1 =xy^2(x^2 - 1)\frac{\partial }{\partial x}  \frac{\partial }{\partial y}  
+y^3\Big(-3x^2 + y(-x^4 + 2x^3 + 2x^2 - 1) + 1\Big)\Big( \frac{\partial }{\partial y}  \Big)^{2},\\
 &\bullet \ \  \theta_2 =  
xy^2(x^2 - 1)\frac{\partial }{\partial x}  \frac{\partial }{\partial y}  
+y^2\Big(xy^2(x + 2) - y(x + 1)^2 + 1\Big)\Big( \frac{\partial }{\partial y}  \Big)^{2},\\
&\bullet \ \  \theta_3=x(x^3 - 2x^2 - x + 2)\Big( \frac{\partial }{\partial x}   \Big)^{2} \\
&+y\Big(-2nx^3 + 6x^2 + 2x(n - 1) -2 + y\big(nx(n - 6) + x^3(-n^2 + 6n - 4) + 2\big) \Big)\frac{\partial }{\partial x}  \frac{\partial }{\partial y} \\
&+y^2\Big(nx(nx + 2n - 6) + 2n +1 + y\big(-n^2(x+1)^2 + y(n^2 + 6nx - 2n - 1)\big) \Big)\Big( \frac{\partial }{\partial y}  \Big)^{2}. 
\end{align*}
Then $\pi^{*}\theta_1$, $\pi^{*}\theta_2$ and $ \pi^{*}\theta_3$ extend to holomorphic global sections of $\Sym^{2}  T_{X}$.
For a general point $q \in W_1$, $\theta_1(q)$, $\theta_2(q)$ and $\theta_3(q)$ give basis of $\Sym^{2}  T_{W_1, q}$. 
Therefore $\Sym^{2} T_{ X}$ is generically globally generated.
\end{proof}

\newpage


\end{document}